\documentclass[12pt,reqno]{amsart}

\usepackage{geometry,graphicx}
\usepackage{amsmath,amssymb} 
\usepackage{hyperref}
\usepackage{bbm}
\usepackage{color}
\usepackage{tikz}
\usepackage{tikz-cd}
\usepackage{mathtools}
\usepackage{parskip}
\usepackage{graphicx}
\tikzset{
  LabelStyle/.style = { rectangle, rounded corners, draw,
                        minimum width = 1em, 
                        font =  },
  VertexStyle/.append style = { inner sep=3pt,
                                font = \large},
  EdgeStyle/.append style = {->, bend left} }

\tikzset{
    ultra thick/.style={line width=5.0pt}
}

\theoremstyle{plain}
\newtheorem{theorem}{Theorem}[section]
\newtheorem{prop}[theorem]{Proposition}
\newtheorem{lemma}[theorem]{Lemma}
\newtheorem{coro}[theorem]{Corollary}

\theoremstyle{definition}
\newtheorem{definition}[theorem]{Definition}
\newtheorem{remark}[theorem]{Remark}
\newtheorem{question}[theorem]{Question}
\newtheorem{example}[theorem]{Example}

\newcommand{\Z}{{\mathbb Z}}

\newcommand{\N}{{\mathbb N}}

\newcommand{\mc}{\mathcal}

\newcommand{\Per}{\operatorname{Per}}

\newcommand{\sub}{\theta}

\newcommand{\rsub}{\vartheta}

\newcommand{\Aut}{\operatorname{Aut}}
\newcommand{\Sym}{\operatorname{Sym}}

\begin{document}

\title[Automorphism groups of RS-subshifts]{
Automorphism groups of random substitution subshifts
}

\author[
R.~Fokkink, D.~Rust, V.~Salo
]{
Robbert Fokkink, Dan Rust, Ville Salo
}
\date{\today}

\address{Institute of Applied Mathematics, Delft University of Technology,\newline
\hspace*{\parindent}Mourikbroekmanweg 6, 2628 XE, Delft, The Netherlands
}
\email{r.j.fokkink@tudelft.nl}

\address{School of Mathematics and Statistics, The Open University, \newline
\hspace*{\parindent}Walton Hall, Milton Keynes, MK7 6AA, UK}
\email{dan.rust@open.ac.uk}

\address{Department of Mathematics and Statistics, University of Turku,\newline \hspace*{\parindent}FI-20014 Turku, Finland}
\email{vosalo@utu.fi}

\keywords{automorphisms, random substitutions, topological conjugacy, amenability}

\subjclass[2010]{37B10, 37A50, 37B40, 52C23}

{\centering\emph{Dedicated to the memory of Uwe Grimm}\par}

\begin{abstract}
We prove that for a suitably nice class of random substitutions, their corresponding subshifts have automorphism groups that contain an infinite simple subgroup and a copy of the automorphism group of a full shift.
Hence, they are countable, non-amenable and non-residually finite.
To show this, we introduce the concept of shuffles and generalised shuffles for random substitutions, as well as a local version of recognisability for random substitutions that will be of independent interest.
Without recognisability, we need a more refined notion of recognisable words in order to understand their automorphisms.
We show that the existence of a single recognisable word is often enough to embed the automorphism group of a full shift in the automorphism group of the random substitution subshift.
\end{abstract}

\maketitle

\section{Introduction}\label{SEC:intro}
The study of automorphism groups of subshifts has a rich history \cite{BLR:auto-sft,H:curtis-hedlund-lyndon,KR:automorphisms}.
The general philosophy is that algebraic properties of the automorphism group can reveal dynamical properties of the subshift.
In particular, the larger the automorphism group, the more `complex' should be the subshift.
This is by no means a hard-and-fast rule, but broadly rings true; subshifts with positive entropy typically have large automorphism groups (for instance, often containing every finite group), while the automorphism groups of subshifts with low subword complexity are virtually $\Z$ \cite{CK:auto-linear,CK:auto-sub-quad,DDMP:auto-low-complexity}.
As well as proving new results, Pavlov and Schmieding give a good survey of the state of the art in the low complexity setting \cite{PS:autos-low-complexity}.

Random substitutions are a relatively new object of study in symbolic dynamics, generalising the notion of substitutions (here called \emph{deterministic substitutions}).
They are of particular interest in the study of mathematical quasicrystals, first studied in this context by Godr\`{e}che and Luck \cite{GL:random} in the 80s, as they provide a mechanism for generating quasicrystals with both positive entropy and long-range order, in the sense that they can have a non-trivial pure point component in their diffraction spectrum.
Baake and Grimm's series \emph{Aperiodic Order} does well to explain this significance \cite[Sec.\ 11.2.3]{BG:book}.
For a gentle introduction to random substitutions in the context of symbolic dynamics, see the foundational paper of Rust and Spindeler \cite{RS:random}.

Whereas deterministic substitutions such as the Fibonacci substitution $a \mapsto ab,\: b \mapsto a$ specify a single image for each letter in the alphabet, a random substitution allows for multiple independent images.
So, for instance, the \emph{random Fibonacci substitution} is given by $a \mapsto \{ab, ba\},\: b \mapsto \{a\}$, where each letter in a word can be independently mapped to one of its possible realisations under the random substitution.
By considering all possible realisations under repeated applications of a random substitution, a language is produced and hence a subshift may be defined.

Under mild assumptions, the corresponding subshift is transitive and has positive entropy.
The subshift retains many of the hierarchical properties of its deterministic counterpart (especially when the random substitution is \emph{recognisable}), but positive entropy and a dense collection of minimal subsets mean that they also mirror aspects of more `complex' subshifts, such as shifts of finite type.
The study of random substitutions is still somewhat in its infancy, however aspects that have been studied already include their topological entropy \cite{G:entropy}, measure theoretic entropy \cite{GMRS:entropy}, ergodic measures \cite{GS:ergodic}, periodic points \cite{R:random-periodic}, topological mixing \cite{EMM:noble-means,MRST:mixing}, diffraction \cite{BSS:rand-diffraction} and connections to SFTs \cite{GRS:random-sft}.
Here, we begin the first investigation of their automorphism groups.
We focus on subgroups that can appear within the automorphism group and conditions that are necessary for such subgroups to appear.

Our main results concern the structure of certain subgroups of the automorphism group.
Under some basic assumptions, we show that the automorphism group of a random substitution subshift contains an infinite simple subgroup (Theorem \ref{THM:alternating-shuffle-simple}), as well as the automorphism group of the full $2$-shift (Theorem \ref{THM:auto-full-shift}).
Section \ref{SEC:prelim} introduces the necessary notation and preliminaries on random substitution subshifts and automorphism groups of subshifts.
In Section \ref{SEC:shuffles}, we introduce the concept of \emph{shuffles} and the shuffle group, which is an important class of automorphisms on which our main methods are based.
Many of our methods rely on a notion of recognisability for random substitutions, which has been previously hinted at in previous work \cite{R:random-periodic}, but which we fully investigate and exploit in Section \ref{SEC:recog}; most notably proving an equivalence between recognisability and local recognisability for compatible random substitutions (Proposition \ref{PROP:local-global}) in Section \ref{SEC:local-recog}.
Recognisability is less ubiquitous in the random setting than in the deterministic setting (even the random Fibonacci substitution is not recognisable).
So, it is important to study random substitutions that only satisfy weaker notions of recognisability.
We introduce the concept of a word being recognisable.
The existence of a single recognisable word is often enough to induce a weak hierarchical structure.
Hence, when the random substitution is of constant length and admits at least one recognisable word in its language, we are still able to show in Section \ref{SEC:const-length} that the automorphism group of a full shift embeds in the automorphism group of the random substitution subshift (Theorem \ref{THM:const-length-full-shift}).

\section{Notation and preliminaries}\label{SEC:prelim}

Let $\mc A = \{a_1, \ldots, a_d\}$ be a finite alphabet with $d$ letters, let $\mc A^n$ denote the set of length-$n$ words over $\mc A$, and let $|u| \coloneqq n$ denote the length of the word $u \in \mc A^n$.
Let $\mc A^+ = \bigcup_{n=1}^\infty \mc A^n$ and let $\mc A^\ast = \mc A^+ \cup \{\epsilon\}$, where $\epsilon$ is the \emph{empty word}.
Let $\mc A^\Z$ denote the \emph{full shift} over $\mc A$ with the usual product topology and for $x \in \mc A^\Z$ and $j \geq i$, write $x_{[i,j]} \coloneqq x_i \cdots x_j$.
Let $\sigma \colon \mc A^\Z \to \mc A^\Z$ denote the usual shift action given by $\sigma(x)_i := x_{i+1}$.
If $u = u_0 \cdots u_{n}$ and $v = v_0 \cdots v_{m}$ are words in $\mc A^\ast$, then their \emph{concatenation}
\[
uv = u_0 \cdots u_{n} v_0 \cdots u_{m}
\]
is a word in $\mc A^{\ast}$ and the set $\mc A^\ast$ forms a monoid under concatenation with the identity given by the empty word $\epsilon$.
If $u \in \mc A^{n}$ is a finite word and $v$ is a finite or (bi-)infinite word such that
$u_i = v_{k+i}$ for some $k$ and all $0 \leq i \leq {n-1}$, then we say that $u$ is a \emph{subword} of $v$ and write $u \triangleleft v$.
If $k = 0$ then we call $u$ a \emph{prefix} of $v$, and if $k = |v|-n$ then we call $u$ a \emph{suffix} of $v$.

\subsection{Random substitutions}
Before defining random substitutions, some attention should be granted to our nomenclature.
In the literature, random substitutions appear in two technically distinct guises; a measure theoretic formalism in which probabilities of image-words are considered, and a topological formalism in which they are not (and which is independent of any chosen non-degenerate probabilities).
We are principally concerned here with random substitutions as a purely topological or combinatorial object.
If we were to consider questions such as the group of measure isomorphisms or any ergodic theoretical questions, this topological setup would be insufficient and probabilities would need to be considered.
Some authors such as Gohlke and Spindeler \cite{GS:ergodic} choose to differentiate the two formalisms by reserving the term `random substitution' only for the measure theoretic formalism and instead referring to the objects appearing in this article as `multi-valued' or `set-valued' substitutions.
For historical reasons \cite{RS:random} and for the sake of brevity, we adopt the slightly ambiguous terminology.

If $U,V \subseteq \mc A^\ast$ are two non-empty sets of words, then we define their concatenation to be the set
\[
UV \coloneqq \{uv \mid u \in U, v \in V\} \subset \mc A^\ast
\]
with the obvious generalisation to the concatenation of finitely many and countably many sets of words.
Let $\mc S$ denote the set of non-empty finite subsets of $\mathcal{A}^+$.
A \emph{random substitution} is a function $\rsub \colon \mc A \to \mc S$.
We extend $\rsub$ to words by defining $\rsub(u_0\cdots u_n) \coloneqq \rsub(u_0)\cdots \rsub(u_n)$ and to finite sets of words $\rsub \colon \mc S \to \mc S$ by defining $\rsub(U) \coloneqq \{\rsub(u) \mid u \in U\}$.
In this way, we may define iterated powers of $\rsub$ by $\rsub^0 \coloneqq \operatorname{Id}_{\mc S}$ and $\rsub^{p+1} \coloneqq \rsub \circ \rsub^p$.

As an example, on the alphabet $\{a,b\}$, the \emph{random Fibonacci substitution} is given by
\[
\rsub \colon a \mapsto \{ab, ba\},\quad b \mapsto \{a\}
\]
with the second and third iterate of the random substitution on the letter $a$ being
\[
\rsub^2(a) = \{aba, baa, aab\},\:\:  \rsub^3(a) = \{abaab, ababa, baaab, baaba, aabab, aabba, abbaa, babaa\}.
\]

A word $u \in \rsub(v)$ is called a \emph{realisation} of the random substitution $\rsub$ on $v$.
Similarly, we call a bi-infinite sequence $x \in \rsub(y)$ a \emph{realisation} of the random substitution $\rsub$ on the bi-infinite sequence $y \in \mc A^\Z$.
If $u \in \vartheta^p(a)$ for some letter $a \in \mc A$, then we call $u$ a \emph{level-$p$ inflation word} or \emph{level-$p$ supertile}.

Let $U \in \mc S$ be a set of words.
If there exists a word $v \in U$ such that $u \triangleleft v$, then we write $u \blacktriangleleft U$.
We say that a word $u \in \mc A^+$ is \emph{$\rsub$-legal} or just \emph{legal} if there exists a letter $a \in \mc A$ and a power $p \geq 0$ such that $u \blacktriangleleft \rsub^p(a)$.
We let $\mc L_\rsub \subset \mc A^\ast$ denote the set of all $\rsub$-legal words and call $\mc L_\rsub$ the \emph{language} of $\rsub$.
We let $\mc L^n_\rsub := \mc L_\rsub \cap \mc A^n$ denote the set of $\rsub$-legal words of length $n$.

The \emph{random substitution subshift} (\emph{RS-subshift}) of $\rsub$, written as $X_\rsub$, is the set of bi-infinite words in $\mc A^\Z$ whose subwords are all $\rsub$-legal. That is,
\[
X_\rsub \coloneqq \{ x \in \mc A^\Z \mid u \triangleleft x \implies u \in \mc L_\vartheta\}.
\]
The RS-subshift $X_\rsub$ associated to a random substitution $\rsub$ is a closed, shift-invariant subspace of the full shift $\mc A^\Z$.
The random substitution $\rsub$ extends to a set-valued function $X_\rsub \rightrightarrows X_\rsub$, or more properly a function $X_\rsub \to 2^{X_\rsub}$, which extends by taking unions to a function $2^{X_\rsub} \to 2^{X_\rsub}$ that can be iterated.

Let $|u|_v$ denote the number of (possibly overlapping) appearances of the word $v$ as a subword of $u$ and let $\psi \colon \mc A^\ast \to \N_0^d$ denote the \emph{abelianisation function} $\psi \colon u \mapsto (|u|_{a_1}, \ldots, |u|_{a_d})^T$.

\begin{definition}
Let $\rsub$ be a random substitution on the alphabet $\mc A$.
We call $\rsub$ \emph{compatible} if for all $a \in \mc A$ and every pair of words $u, v \in \rsub(a)$, $\psi(u) = \psi(v)$.
The \emph{substitution matrix} $M_\rsub$ associated with a compatible random substitution is given by $(M_\rsub)_{ij} = |\rsub(a_j)|_{a_i}$, which is well defined by compatibility.
\end{definition}
\begin{definition}
Let $\rsub$ be a random substitution on the alphabet $\mc A$.
We call $\rsub$ \emph{primitive} if there exists a power $p \geq 1$ such that for all $a,b \in \mc A$, $a \blacktriangleleft \rsub^p(b)$.
\end{definition}
If $\rsub$ is compatible, then $\rsub$ is primitive if and only if $M_\rsub$ is a primitive matrix.
That is, if $M_\rsub^p$ has positive entries for some power $p \geq 1$.
For the majority of this article, we will only consider compatible primitive substitutions.

A random substitution $\rsub \colon \mc A \to \mc S$ is said to be \emph{constant length} with length $L$ if, for every $a \in \mc A$ and $u \in \rsub(a)$, $|u| = L$.
An example of a compatible random substitution of constant length is the \emph{random period doubling} substitution $\rsub \colon a \mapsto \{ab,ba\},\: b \mapsto \{aa\}$ with constant length $L = 2$, first introduced by Baake, Spindeler and Strungaru \cite{BSS:rand-diffraction}.
\begin{definition}\label{DEF:recog}
Let $\rsub$ be a compatible random substitution on the alphabet $\mc A$.
We call $\rsub$ \emph{recognisable} if for every $x \in X_\rsub$, there exists a unique $y \in X_\rsub$ and a unique $0 \leq k \leq |\rsub(y_0)| - 1$ such that $\sigma^{-k}(x) \in \rsub(y)$.
\end{definition}
Note that the important part of Definition \ref{DEF:recog} is the uniqueness of the pair $(y,k)$ and not existence, which is guaranteed by \cite[Lem. 12]{RS:random} (in the primitive setting).

Rust showed that recognisability of $\rsub$ implies aperiodicity of $X_\rsub$ \cite{R:random-periodic}.
That is, if $\rsub$ is recognisable, then the RS-subshift $X_\rsub$ contains no periodic points.
The converse is in general not true, as the random Fibonacci substitution $a \mapsto \{ab,ba\}, b \mapsto \{a\}$ is not recognisable but, by considering letter frequencies, its RS-subshift is aperiodic.

While recognisability is a useful property, and the above global definition is theoretically useful, it is often difficult to directly verify for examples.
Instead, one generally checks an equivalent local version of recognisability.
This will be introduced and studied in Section \ref{SEC:local-recog}, where the local and global definitions will be shown to be equivalent.
Until then, without loss of rigour, we interpret recognisability of a compatible random substitution as meaning that for any $x \in X_\rsub$, there is a unique decomposition of $x$ into inflation words, and the type of each inflation word is also uniquely determined.

It is straightforward to check \cite{R:random-periodic} that $\rsub$ is recognisable if and only if $\rsub^n$ is recognisable for all $n \geq 1$.

\begin{example}
Consider the substitution $\rsub$ on the alphabet $\{a,b\}$ given by
\[
a \mapsto \{abb, bab\}, \quad b \mapsto\{aa\}.
\]
First, notice that we can determine how to decompose any string of consecutive $a$s of length at least two into inflation words, as the start of any such string must begin with an inflation word of type $\rsub(b)$.
So, for instance, the word $baaaaab$ must be decomposed as $b|aa|aa|ab$.
Notice also that every word of length at least twenty four must contain the word $aa$, and so $aa$ occurs with uniformly bounded gaps in any element $x \in X_\rsub$.
As all inflation words of type $\rsub(b)$ are uniquely determined, anything not yet determined must be a concatenation of inflation words of type $\rsub(a)$, which then are also uniquely determined as they both have length three.
It follows that $\rsub$ is recognisable.

So, for example, the word $baaaabababbabbaaab$ can be decomposed by first decomposing the strings of $a$s as
\[
b|aa|aa|bababbabb|aa|ab
\] and then noticing that there is exactly one choice in how the word between the strings of $a$s can be decomposed.
So the entire word has the unique decomposition 
\[
b|aa|aa|bab|abb|abb|aa|ab.
\]
\end{example} 

\begin{example}[random square Fibonacci]\label{EX:fib-squared}
The following example is worth mentioning as it is related to the square of the Fibonacci substitution\footnote{The marginal $a \mapsto aab,\: b\mapsto ab$ is conjugate to the square of the Fibonacci substitution $a \mapsto aba, \: b \mapsto ab$. Hence, the usual Fibonacci subshift is contained in the RS-subshift of $\rsub$.} but has the advantage that, unlike the usual random Fibonacci substitution, this one is recognisable.
Consider the random substitution $\rsub$ on $\{a,b\}$ given by
\[
a \mapsto \{aab\}, b \mapsto \{ab,ba\}.
\]
This example first appeared (reversed) in the work of Nilsson \cite{N:random-two}.
We check recognisability as follows.
Any appearance of $bb$ must mean that the second $b$ comes from the start of an inflation word of type $\rsub(b)$ and all inflation words $ba$ of type $\rsub(b)$ are found in this way.
As $bbb$ is not legal, any appearance of the word $abaab$ can only be decomposed as $ab|aab$. Any appearance of the word $baaab$ can only be decomposed as $ba|aab$, and all inflation words $aab$ of type $\rsub(a)$ are found in either of these two ways.
Anything not yet identified must then be the inflation word $ab$ of type $\rsub(b)$.
Recognisability then follows.
\end{example}
The following condition is a useful consequence of recognisability for compatible random substitutions.
\begin{definition}
A random substitution $\rsub$ satisfies the \emph{disjoint set condition} if for all $a \in \mc A$, $u,v \in \rsub(a)$ and $k \geq 1$, we have $\rsub^k(u) \cap \rsub^k(v) \neq \emptyset \implies u=v$.
\end{definition}
The disjoint set condition always holds for recognisable compatible random substitution \cite[Lem.\ 4.5]{GMRS:entropy}.

\subsection{Automorphisms}\label{SUBSEC:autos}
Let $(X,\sigma)$ and $(Y,\tau)$ be topological dynamical systems.
A function $f \colon X \to Y$ is called a \emph{topological conjugacy} if $f$ is a homeomorphism and $f \circ \sigma = \tau \circ f$.
A topological conjugacy $h \colon X \to X$ is called an \emph{automorphism}.
We denote the set of all automorphism of $(X,\sigma)$ by
\[
\Aut(X,\sigma)\coloneqq \{h \colon X \to X \text{ homeomorphism} \mid h \circ \sigma = \sigma \circ h\},
\]
which forms a group under composition.
If the action is unambiguous, then we write $\Aut(X) \coloneqq \Aut(X,\sigma)$ without loss of generality.

For a subshift $X$, the shift map $\sigma$ is always an automorphism of $X$ and so, if $X$ contains a non-periodic orbit, then $\langle \sigma \rangle \cong \Z$ is an infinite subgroup of $\Aut(X)$. Indeed, $\Aut(X)$ is the centraliser of the element $\sigma$ in the group $\operatorname{Homeo}(X)$ of homeomorphisms of the subshift\footnote{This observation recently inspired Baake, Roberts and Yassawi \cite{BRY:reversing-syms} to consider the \emph{normaliser} of $\sigma$ in $\operatorname{Homeo}(X)$, rather than the centraliser, the so-called group of `reversing symmetries' of the subshift.}.
It is a consequence of the Curtis--Hedlund--Lyndon Theorem \cite{H:curtis-hedlund-lyndon} that $\Aut(X)$ is countable.

Recall that a group $G$ is residually finite if for every $g \in G \setminus \{1\}$, there exists a finite group $H$ and a homomorphism $h \colon G \to H$ such that $h(g) \neq 1$.

It was shown by various authors that for $X$ a mixing shift of finite type, $\Aut(X)$ contains subgroups isomorphic to the following:
\begin{itemize}
\item every countable direct sum of finite groups \cite{BLR:auto-sft},
\item the free group $F_2$ on two generators, hence the free group on countably many generators \cite{BLR:auto-sft},
\item a countable direct sum of the integers $\bigoplus_{i=1}^\infty \Z$ \cite{BLR:auto-sft},
\item finite free products of finite groups \cite{A:free-product-finite},
\item every residually finite countable group that is a union of finite groups \cite{KR:automorphisms},
\item all right-angled Artin groups, including the fundamental group of any surface \cite{KR:automorphisms}.
\end{itemize}
Boyle, Lind and Rudolph further observed that $\Aut(X)$ is residually finite, and therefore any finitely presented subgroup has a solvable word problem (in fact any finitely \emph{generated} subgroup).
$\Aut(X)$ itself is not finitely generated \cite[Theorem 7.8]{BLR:auto-sft}.
Kim and Roush proved in \cite{KR:automorphisms} that the set of subgroups of $\Aut(X)$ is closed under virtual extensions, and Salo proved that they are also closed under cograph products (free and direct products) \cite{S:closed-free-prod} and general countable graph products \cite{S:graph-prod}.
Combining these, we immediately get all the items on the above list and more.

The statement above regarding residually finite automorphism groups is a general statement about subshifts with a dense subset of periodic points.
As far as we are aware, this result is part of the folklore of automorphism groups of subshifts.
A proof for mixing subshifts of finite type appears in the paper of Boyle, Lind and Rudolph \cite[Theorem 3.1]{BLR:auto-sft}, although their proof also readily applies to the general case.
As this will play a role later, we provide here a short proof of this statement in order to remain as self-contained as possible.
\begin{prop}\label{PROP:periodic}
Let $X$ be a subshift such that the set of periodic points $\Per(X)$ is dense in $X$.
The automorphism group $\Aut(X)$ is residually finite.
\end{prop}
\begin{proof}
The set of periodic points $P \coloneqq \Per(X)$ is shift-invariant.
An automorphism of $X$ induces an automorphism of $P$.
Since $P$ is dense, the restriction $\Aut(X) \to \Aut(P)$ is injective and so $\Aut(X)$ embeds as a subgroup of $\Aut(P)$.
It therefore suffices to show that $\Aut(P)$ is residually finite.

Let $P(n)$ be the set of points of period $n$ in $X$.
There is an induced map $f_n \colon \Aut(P) \to \Aut(P(n))$ given by restriction.
Note that $\Aut(P(n))$ is finite, as $P(n)$ is finite.
For every automorphism $g \in \Aut(P)\setminus\{1\}$, there is a natural number $n$ such that $g(x) \neq x$ for some $x \in P(n)$.
Hence, $f_n(g) \neq 1$, and so $\Aut(P)$ is residually finite.
\end{proof}

\section{The shuffle group}\label{SEC:shuffles}
In this section, we introduce an important building block and precursor to many of the automorphisms that we will construct on RS-subshifts, so-called \emph{shuffles}.
Deterministic substitutions are also random substitutions, however they are precisely the substitutions for which shuffles will turn out to always be trivial---deterministic substitutions are also extremely well-studied.
We therefore restrict our attention to the study of subshifts coming from properly random substitutions, and we treat the deterministic case as `degenerate'.
\begin{definition}
A random substitution $\rsub$ is \emph{degenerate} if $\#\rsub(a) = 1$ for all $a \in \mc A$ and \emph{non-degenerate} otherwise.
A random substitution $\rsub$ is \emph{$k$-branching on $a \in \mc A$} if there exists $p \geq 1$ with $\#\rsub^p(a) \geq k$, is \emph{branching on $a$} if $\rsub$ is $k$-branching on $a$ for all $k \geq 1$, and is \emph{branching} if $\rsub$ is branching on $a$ for each $a \in \mc A$.
\end{definition}
It turns out that one requires only very weak conditions in order to have branching.
\begin{prop}\label{PROP:branching}
If $\rsub$ is a non-degenerate, primitive random substitution and $X_\rsub$ is non-empty, then $\rsub$ is branching.
\end{prop}
We leave the proof as an exercise to the reader, but note that when $X_\rsub$ is non-empty, then by results in \cite{RS:random}, there exists a letter $a \in \mc A$ such that $\rsub(a)$ contains a word of length greater than 1.

\begin{remark}
Primitivity is only a sufficient condition in the above proposition, as the non-primitive branching substitution $a \mapsto \{ab, ba\}, \: b \mapsto \{a\}, \: c \mapsto \{aca\}$ illustrates.
\end{remark}

\begin{definition}
Let $\rsub$ be a recognisable compatible random substitution and let $x \in X_\rsub$ be an element of the RS-subshift.
By recognisability, for some $0 \leq k \leq |\rsub^n(y_0)|-1$, we may unambiguously write $x$ as a bi-infinite concatenation of level-$n$ inflation words as
\[
\sigma^{k}\left(\cdots u_{-1}.u_0 u_1 \cdots \right)  \quad \in \quad  \sigma^{k}\left(\cdots\rsub^n(y_{-1}).\rsub^n(y_0)\rsub^n(y_1) \cdots\right).
\]
Let $a \in \mc A$ and let $\alpha \in \Sym(\rsub^n(a))$ be a permutation of the set of level-$n$ inflation words $\rsub^n(a)$.
We define the map $f_\alpha \colon X_\rsub \to X_\rsub$ by applying the permutation $\alpha$ to each $u_i$ of type $\rsub^n(a)$ in the above decomposition of $x$.
\end{definition}
Note that such a map is well-defined for recognisable substitutions and it is clearly an invertible function that commutes with the shift action. Continuity of $f_\alpha$ follows from the equivalence of recognisability with a local version of recognisability that will be shown in Section \ref{SEC:local-recog}.
It follows that $f_\alpha$ is an automorphism of $X_\rsub$.
So, the set
\[
\Gamma_{n,a} \coloneqq \{f_\alpha \colon X_\rsub \to X_\rsub \mid \alpha \in \Sym(\rsub^n(a))\}
\] forms a group of automorphisms of $X_\rsub$ isomorphic to the symmetric group $S_{\#\rsub^n(a)}$. We similarly define
\[
\Gamma_{n} \coloneqq \langle \Gamma_{n,a} \mid a \in \mc A \rangle = \prod_{a \in \mc A} \Gamma_{n,a}, \qquad \Gamma \coloneqq \bigcup_{n \geq 1} \Gamma_n,
\]
and call $\Gamma$ the \emph{shuffle group} of $\rsub$.
An element $f_\alpha \in \Gamma_n$ is called a \emph{level-$n$ shuffle}.
Note that $\Gamma_{n}$ is a subgroup of $\Gamma_{n+1}$. This is a consequence of the disjoint set condition, as any set of permutations of the sets $\rsub^k(a_i)$, $a_i \in \mc A$ induces (by concatenation) a well-defined permutation of the sets $\rsub^k(u)$ for $u \in \rsub(a_j)$, and any word $v \in \rsub^{k+1}(a_j)$ belongs to exactly one of the sets $\rsub^{k}(u)$ for $u \in \rsub(a_j)$. Hence, a well-defined permutation is defined on the whole of $\rsub^{k+1}(a_j)$ for each $a_j \in \mc A$. So, $\Gamma$ is an infinite union of nested subgroups of $\Aut(X_\rsub)$. It follows that $\Gamma$ is also a subgroup of $\Aut(X_\rsub)$.

\begin{example}\label{EX:shuffle}
Let $\rsub$ be the random square Fibonacci substitution given by
\[
\rsub \colon a \mapsto \{aab\}, b \mapsto \{ab,ba\}.
\]
Write $u_1 = ab$ and $u_2 = ba$.
We already showed that $\rsub$ is recognisable in Example \ref{EX:fib-squared}.

Let $\alpha$ be the non-trivial element of $\Sym(\rsub(b))$. The shuffle $f_\alpha \in \Gamma_{1,b}$ is the one that replaces all occurrences of $u_1$ in an element of $X_\rsub$ by $u_2$ and vice versa. So, for instance, the action of $f_\alpha$ maps in the following way
\begin{center}
\begin{tikzpicture} 
\node (1) at (0,0) {$\cdots \overbracket{aab}^{a} \: \overbracket{aab}^{a} \: \overbracket{ab}^{b} \: \overbracket{aab}^{a} \: \overbracket{aab}^{a} \: \overbracket{ba}^{b} \: \overbracket{aab}^{a} \: \overbracket{ba}^{b} \: \overbracket{aab}^{a} \: \overbracket{aab}^{a} \: \overbracket{ba}^{b} \: \overbracket{ab}^{b} \: \overbracket{aab}^{a} \: \overbracket{ba}^{b}\: \overbracket{aab}^{a}\: \overbracket{aab}^{a}\: \overbracket{aab}^{a}\: \overbracket{ab}^{b} \cdots\phantom{,}$};
\node (2) at (0,-1.8) {$\cdots \underbracket{aab}_{a} \: \underbracket{aab}_{a} \: \underbracket{ba}_{b} \: \underbracket{aab}_{a} \: \underbracket{aab}_{a} \: \underbracket{ab}_{b} \: \underbracket{aab}_{a} \: \underbracket{ab}_{b} \: \underbracket{aab}_{a} \: \underbracket{aab}_{a} \: \underbracket{ab}_{b} \: \underbracket{ba}_{b} \: \underbracket{aab}_{a} \: \underbracket{ab}_{b}\: \underbracket{aab}_{a}\: \underbracket{aab}_{a}\: \underbracket{aab}_{a}\: \underbracket{ba}_{b} \cdots.$}; 

\node (1a) at (-4.46,-0.35) {};
\node (1b) at (-2.34,-0.35) {};
\node (1c) at (-1.03,-0.35) {};
\node (1d) at (1.07,-0.35) {};
\node (1e) at (1.62,-0.35) {};
\node (1f) at (2.95,-0.35) {};
\node (1g) at (5.83,-0.35) {};

\node (2a) at (-4.46,-1.5) {};
\node (2b) at (-2.34,-1.5) {};
\node (2c) at (-1.03,-1.5) {};
\node (2d) at (1.07,-1.5) {};
\node (2e) at (1.62,-1.5) {};
\node (2f) at (2.95,-1.5) {};
\node (2g) at (5.83,-1.5) {};

\draw [->,dashed,thick] (1a) -- (2a);
\draw [->,dashed,thick] (1b) -- (2b);
\draw [->,dashed,thick] (1c) -- (2c);
\draw [->,dashed,thick] (1d) -- (2d);
\draw [->,dashed,thick] (1e) -- (2e);
\draw [->,dashed,thick] (1f) -- (2f);
\draw [->,dashed,thick] (1g) -- (2g);
\end{tikzpicture}
\end{center}

A shuffle can be thought of as a local rule on inflation words which can only see as far as the nearest inflation word.
It is therefore also natural to consider \emph{generalised shuffles} that can see further than just the nearest inflation word, in order to give rise to more interesting automorphisms (in analogue to sliding block codes with radius larger than 0). For instance, consider the following example, where the local rule also depends on the next nearest inflation word to the right.
Define an automorphism $f$ that replaces an occurrence of $u_i$ with $u_j$ where $u_j$ is the inflation word of type $\rsub(b)$ that is immediately to the right of $u_i$.
For instance, the same element as before is mapped in the following way under the action of $f$
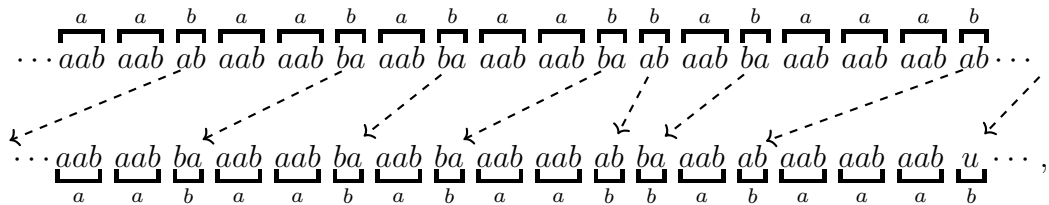
\begin{center}
\begin{tikzpicture} 
\node (1) at (0,0) {$\cdots \overbracket{aab}^{a} \: \overbracket{aab}^{a} \: \overbracket{ab}^{b} \: \overbracket{aab}^{a} \: \overbracket{aab}^{a} \: \overbracket{ba}^{b} \: \overbracket{aab}^{a} \: \overbracket{ba}^{b} \: \overbracket{aab}^{a} \: \overbracket{aab}^{a} \: \overbracket{ba}^{b} \: \overbracket{ab}^{b} \: \overbracket{aab}^{a} \: \overbracket{ba}^{b}\: \overbracket{aab}^{a}\: \overbracket{aab}^{a}\: \overbracket{aab}^{a}\: \overbracket{ab}^{b} \cdots\phantom{,}$};
\node (2) at (0,-1.8) {$\cdots \underbracket{aab}_{a} \: \underbracket{aab}_{a} \: \underbracket{ba}_{b} \: \underbracket{aab}_{a} \: \underbracket{aab}_{a} \: \underbracket{ba}_{b} \: \underbracket{aab}_{a} \: \underbracket{ba}_{b} \: \underbracket{aab}_{a} \: \underbracket{aab}_{a} \: \underbracket{ab}_{b} \: \underbracket{ba}_{b} \: \underbracket{aab}_{a} \: \underbracket{ab}_{b}\: \underbracket{aab}_{a}\: \underbracket{aab}_{a}\: \underbracket{aab}_{a}\: \underbracket{\,u\:}_{b} \cdots,$}; 

\node (1a) at (-4.46,-0.35) {};
\node (1b) at (-2.34,-0.35) {};
\node (1c) at (-1.03,-0.35) {};
\node (1d) at (1.07,-0.35) {};
\node (1e) at (1.62,-0.35) {};
\node (1f) at (2.95,-0.35) {};
\node (1g) at (5.83,-0.35) {};
\node (1h) at (6.83,-0.35) {};

\node (20) at (-7,-1.4) {};
\node (2a) at (-4.46,-1.4) {};
\node (2b) at (-2.34,-1.4) {};
\node (2c) at (-1.03,-1.4) {};
\node (2d) at (1.07,-1.4) {};
\node (2e) at (1.62,-1.4) {};
\node (2f) at (2.95,-1.4) {};
\node (2g) at (5.83,-1.4) {};

\draw [->,dashed,thick] (1a) -- (20);
\draw [->,dashed,thick] (1b) -- (2a);
\draw [->,dashed,thick] (1c) -- (2b);
\draw [->,dashed,thick] (1d) -- (2c);
\draw [->,dashed,thick] (1e) -- (2d);
\draw [->,dashed,thick] (1f) -- (2e);
\draw [->,dashed,thick] (1g) -- (2f);
\draw [->,dashed,thick] (1h) -- (2g);
\end{tikzpicture}
\end{center}
where $u$ depends on unknown information to the right.
The inverse is then given by replacing each $u_i$ with the inflation word to its immediate left.
Continuity of $f$ is clear, as is the fact that $f$ commutes with the shift.
Indeed, every word of length $12$ contains at least one occurrence of a type-$b$ inflation word, so $f$ is locally defined (by the local version of recognisability that will appear in Section \ref{SEC:local-recog}) and describes a sliding block code.
In essence, this automorphism keeps all type-$a$ inflation words fixed, but `shifts' the type-$b$ inflation words to the left.
Notice that shuffles have finite order, whereas this generalised shuffle $f$ has infinite order.
Aperiodicity of $X_\rsub$ follows from recognisability of $\rsub$.
As $\mc A$ has only two letters, that means the $a$s, and hence also the inflation words of type $\rsub(a)$, are aperiodically positioned and so this automorphism $f$ is not a power of the shift $\sigma$.
Nor is $f$ contained in $\langle \sigma \rangle \Gamma$, the subgroup of $\Aut(X)$ generated by shuffles and shifts.
\end{example}
\begin{remark}\label{REM:gen-shuff}
While a rigorous definition of a generalised shuffle can be given, doing so would be notation-heavy and cumbersome. Instead, we describe them intuitively as the analogue of a finite-radius sliding block code at the level of inflation words of a particular type $\rsub^k(a)$, where the symbols for the sliding block code are given by the elements of the set $\rsub^k(a)$. That is, they are locally defined maps that replace $\rsub^k(a)$-inflation words with a possibly different realisation according to the local configuration of all $\rsub^k(a)$-inflation words out to some fixed radius. We hope that this description and the above illustrative example sufficiently convinces the reader that such maps induce well-defined continuous functions $X_\rsub \to X_\rsub$ that commute with the shift action.

In order for such a map to be well-defined we need that $\rsub$ is recognisable, so that it makes sense to apply the generalised shuffle at the level of inflation words in a unique way. For the induced map to be continuous, we further require that the letter $a$ appears with uniformly bounded gaps in every element of $X_\rsub$ (namely, that for some $C \geq 1$, every legal word of length $C$ contains the letter $a$), so that the corresponding map is uniformly locally defined and we therefore obtain a corresponding sliding block code. This final point will be addressed in Lemma \ref{LEM:uni-bound-gaps} for primitive, compatible random substitutions.
\end{remark}

With the previous example in mind, it is not so surprising that automorphism groups of RS-subshifts are similar in many respects to the automorphism groups of full shifts.
As we shall see later, this pseudo-embedding of the shift on $\{0,1\}^\Z$ into $\Aut(X_\rsub)$ can be generalised to any automorphism on $\{0,1, \ldots, n\}^\Z$. For now though, we concentrate just on the shuffle group, which has its own rich structure.

Clearly, for a degenerate random substitution, the shuffle group is trivial.
Conversely, the shuffle group for a branching substitution is always non-trivial.
In fact, if a random substitution $\rsub$ is branching on $a$, then as $\#\rsub^n(a)$ can be arbitrarily large, and $\Gamma_{n,a}$ is a symmetric group on $\#\rsub^n(a)$ elements, Cayley's theorem tells us that every finite group $G$ embeds into $\Gamma_{n,a}$ for some $n \geq 1$, $a \in \mc A$.
\begin{prop}
Let $\rsub$ be a recognisable, branching, compatible, random substitution.
Any finite group $G$ embeds into the shuffle group $\Gamma$. \qed
\end{prop}
\begin{remark}
Given the above discussion, we now make the standing assumption that all random substitutions under consideration are branching and recognisable, unless stated otherwise.
\end{remark}
Much more can be said about the structure of the shuffle group.
Before introducing general results, we begin with a particularly illuminating example, which showcases some of the more varied subgroups that can be found within the shuffle group.

\begin{example}\label{EX:thompson}
Let $\rsub$ be a random substitution on the alphabet $\mc A = \{0, 1, 2\}$ given by
\[
\rsub \colon 0 \mapsto \{012, 210\},\quad 1 \mapsto \{120, 021\},\quad 2 \mapsto \{201, 102\}.
\]
Notice that $\rsub$ is invariant under cyclic permutation of letters.
That is, if we define $\overline{a} \coloneqq a+1 \pmod 3$ and we extend this to words $u$, so that $\overline{u_1 \cdots u_k} \coloneqq \overline{u_1} \cdots \overline{u_k}$, then $u \in \rsub(a) \iff \overline{u} \in \rsub(\overline{a})$.
Further, this extends to powers, so $u \in \rsub^k(a) \iff \overline{u} \in \rsub^k(\overline{a})$.

This random substitution is primitive and compatible.
To see that $\rsub$ is recognisable, let
\[
 \cdots \:c_{-2}\underbracket{a_{-1}b_{-1}c_{-1}}_{\rsub(y_{-1})} \underbracket{a_{0}\:b_{0}\:c_{0}\:}_{\rsub(y_{0})} \underbracket{a_{1}\:b_{1}\:c_{1}}_{\rsub(y_{1})}\, a_2\:b_2\: \cdots
\]
be a decomposition of $x \in X_\rsub$, with preimage $y \in X_\rsub$.
If there is another decomposition of $x$, then by the fact that $\rsub$ is constant length, the borders of inflation words are offset everywhere from the above decomposition by either $1$ or $2$ to the right.
Without loss of generality, suppose that the offset is $1$.
That is, there is another decomposition of $x$ that looks like
\[
 \cdots \:c_{-2}a_{-1}\underbracket{b_{-1}c_{-1}a_{0}\:}_{\rsub(z_{-1})} \underbracket{b_{0}\:c_{0}\:a_{1}\:}_{\rsub(z_{0})} \underbracket{b_{1}\:c_{1}\:a_{2}}_{\rsub(z_{1})}\, b_2 \: \cdots
\]
for some other preimage $z \in X_\rsub$.

Observe that all inflation words contain exactly one copy of each letter $0,1,2$ and so it follows by induction that $a_i = a_j$ for all $i,j \in \Z$.
Indeed, $a_ib_ic_i$ contains one of each letter as it is an inflation word, and so if $b_ic_ia_{i+1}$ is also an inflation word, then $a_{i+1}$ must be the same as $a_i$.
However, this is impossible because if $a_0 = 0$ for instance, then this would imply that every realisation of $\vartheta(z_i)$ ends with $0$, which cannot be the case if $z_i = 2$ (and infinitely many of the $z_i$ are $2$).
We reach similar contradictions if $a_0 = 1$ or $a_0 = 2$.
It follows that the second decomposition cannot exist and so $x$ has a unique substitutive preimage.
Hence, $\rsub$ is recognisable.

We now define a set of shuffles that respect the symmetry under adding bars.
Let $n \geq 1$ and $\alpha \in \Sym(\rsub^n(0))$.
Given this permutation $\alpha$, define the associated permutations $\overline{\alpha} \in \Sym(\rsub^n(1))$ and $\overline{\overline{\alpha}} \in \Sym(\rsub^n(2))$ by
\[
\overline{\alpha}(\overline{u}) = \overline{\alpha(u)}, \qquad \overline{\overline{\alpha}}(\overline{\overline{u}}) = \overline{\overline{\alpha(u)}}.
\]
Then the automorphism $g_\alpha = f_\alpha \circ f_{\overline{\alpha}} \circ f_{\overline{\overline{\alpha}}}$ is a level-$n$ shuffle.
If we define
\[
G_n \coloneqq \{g_{\alpha} \mid \alpha \in \Sym(\rsub^n(0))\} \leqslant \Gamma_n,
\]
then it is clear that $G_n$ is isomorphic to the symmetric group $S_{\#\rsub^n(0)}$.
We may also define a natural embedding of $G_n$ into $G_{n+1}$.

Note that the number of inflation words in $\rsub^n(0)$ is calculated to be
\[
\prod_{i=0}^{n-1} 2^{3^i} = 2^{a(n)},
\]
where $a(n) = \frac{1}{2}\left(3^{n}-1\right) \in \N$ and so is an integral power of $2$.
Hence, $G_n \cong S_{2^{a(n)}}$. Write $\iota_n \colon G_n \to S_{2^{a(n)}}$ for this isomorphism.
The inclusion $\gamma_n \colon G_n \hookrightarrow G_{n+1}$ is defined as follows.
Let $\alpha \in \Sym(\rsub^n(0))$ and $g_\alpha \in G_n$ and write $\alpha_0 = \alpha$, $\alpha_1 = \overline{\alpha}$ and $\alpha_2 = \overline{\overline{\alpha}}$.
Let $u \in \rsub^{n+1}(0)$.
The word $u$ is a concatenation of three words $u = u_1 u_2 u_3$, where $u_i \in \rsub^n(j(i))$ for $j(i) \in \mc A$.
We define $\beta(u) = \alpha_{j(1)}(u_1)\alpha_{j(2)}(u_2)\alpha_{j(3)}(u_3)$ and then the embedding $\gamma_n$ is given by $g_\alpha \mapsto g_\beta$.

The direct limit
\[
G = \varinjlim\left(G_n \stackrel{\iota_n}{\hookrightarrow} G_{n+1}\right)
\]
has a structure that is similar to an infinite symmetric group, in that they are direct limits of increasing chains of symmetry groups with embeddings as connecting maps given by the maps $\iota_{n+1} \circ \gamma_n \circ \iota_{n}^{-1}\colon S_{2^{a(n)}} \hookrightarrow S_{2^{a(n+1)}}$. In particular, they are locally finite, infinite groups.

Write $\gamma_{nm} = \gamma_{m-1} \circ \cdots \circ \gamma_{n+1} \circ \gamma_n$. Notice that if $\alpha \in S_{2^{a(n)}}$ is an even permutation, then $\gamma_{nm}(g_\alpha) = g_\beta$ for a permutation $\beta \in S_{2^{a(m)}}$ that is also even (this is a general property of homomorphic images of alternating groups in symmetric groups).
Hence, $\gamma_{nm}$ restricts to an embedding of alternating groups $A_{2^{a(n)}} \hookrightarrow A_{2^{a(m)}}$.
The direct limit $H = \varinjlim(A_{2^{a(n)}} \hookrightarrow A_{2^{a(n+1)}})$ associated with these embeddings is a union of simple groups and is therefore itself simple.
Moreover, $H$ is clearly an infinite group, and so $\Gamma$, and hence $\Aut(X_\rsub)$, contains an infinite simple subgroup.
It follows that $\Aut(X_\rsub)$ cannot be residually finite.
\end{example}

We emphasise the remarkable property that $\Aut(X_\rsub)$ contains both the automorphism group $\Aut(X_2)$ of a full shift and an infinite simple subgroup.
Other such examples exist, but of a rather artificial flavour, such as the product of the Thue-Morse subshift $X_{TM}$ and $X_2$ with the diagonal action\footnote{This is a non-trivial result, the proof of which is beyond the scope of this work. Unfortunately, we are not aware of a reference.}.
Whereas, most would agree that $X_\rsub$ is a more natural example.
Of course, it should be noted that $X_{TM} \times X_2$ is in fact an RS-subshift with (non-compatible, non-recognisable) random substitution given by
\[
\rsub \colon a_i \mapsto \{a_0 b_0, a_0 b_1, a_1 b_0, a_1 b_1\}, \quad b_i \mapsto \{b_0 a_0, b_0 a_1, b_1 a_0, b_1 a_1\}, \quad i \in \{0,1\}.
\]

There are few known classes of examples for which the automorphism group contains an infinite simple subgroup.
Perhaps the simplest example is the subshift $X_{\leq 2}$ comprising all sequences in $\{0,1\}$ for which the letter $1$ appears at most twice.
Indeed, Salo and Shraudner \cite{SS:not-res-fin} showed that if a countable subshift $X$ contains an infinite set of isolated points in distinct orbits with the same eventually periodic tails, then $\Z^\omega \rtimes S_\infty$ is a subgroup of $\Aut(X)$. Here, $S_\infty$ is the symmetric group on $\N$, and so it contains the infinite alternating group $A_\infty$; the subgroup of elements that are products of an even number of transpositions.
In $X_{\leq 2}$, such a set of isolated points is given by $\{\cdots 0010^n100\cdots \mid n \geq 0\}$.
However, our subshift $X_\rsub$ is very different, being uncountable and containing no isolated points.
In particular, $X_\rsub$ is even transitive.

\subsection{The alternating shuffle group}
The shuffle group in the last example was shown to contain an infinite simple subgroup, and therefore the automorphism group $\Aut(X_\rsub)$ is not residually finite.
This property applies more generally than just for this example.
To show this, we need to introduce some results from basic group theory.

For the benefit of the reader, we restate Goursat's Lemma \cite[Theorem 5.5.1]{H:Goursat}.

\begin{lemma}[Goursat's Lemma]
Let $G = G_1 \times G_2$ be a product of groups with projection maps $\pi_i$ onto $G_i$.
Let $N$ be a subgroup of $G$ with the property that $\pi_i(N) = G_i$ and write $N_1 = N \cap \ker \pi_2$, $N_2 = N \cap \ker \pi_1$.
Then, the image of $N$ under the quotient $q \colon G \to G_1/N_1 \times G_2/N_2$ is the graph of an isomorphism $\phi \colon G_1/N_1 \to G_2/N_2$.
That is, 
\[
q(N) = \{(g,\phi(g)) \mid g \in G_1/N_1\}.
\]
\end{lemma}
Using Goursat's Lemma, we can prove the following classification result for normal subgroups of finite products of non-abelian simple groups.

\begin{lemma}\label{LEM:factor}
Let $G=G_1\times\cdots\times G_n$ be a product of non-abelian simple groups and let $N\subset G$ be a normal subgroup.
Then $N$ is a product of $G_i$'s.
\end{lemma}
\begin{proof}
We proceed by induction on $n$.
The statement is trivially true if $n=1$.
Suppose the statement holds for $n-1$.
We identify the factors $G_i$ with elements of $G$ that have all coordinates equal to $1$ except
their $i$-th coordinate.
Write $G$ as $G_1\times(G_2\times\cdots\times G_n)=G_1\times \widehat{G}$.
By the inductive hypothesis, a normal subgroup of $\widehat{G}$ is a product of $G_i$'s for $i>1$.
If the projection of $N$ onto $\widehat{G}$ is not surjective, then $N$ is contained in a product of fewer than $n$ simple groups and we are done by induction.
By a similar argument, the projection of $N$ on the first coordinate is also surjective.
Let $K_1$ be the kernel of the projection of $N$ onto the \emph{second} coordinate.
It is a normal subgroup of $G_1$, hence it is either equal to $G_1$ or to $\{1\}$.
In the first case, $G_1$ is a Cartesian factor of $N$ and we are done by induction.

Now, suppose that $K_1=\{1\}$ and let $\widehat{K}$ be the kernel of the projection onto the \emph{first} coordinate.
By Goursat's lemma, the image of $N$ in $G_1\times \widehat{G}/\widehat{K}$ is the graph of an isomorphism $\phi\colon G_1 \to \widehat{G}/\widehat{K}$.
Since $N$ is a normal subgroup of $G_1\times \widehat{G}$, the graph $\{(g,\phi(g))\mid g\in G_1\}$
is normal in $G_1\times \widehat{G}/\widehat{K}$.
In particular, for all $h \in \widehat{G}/\widehat{K}$ and $g \in G_1$, $(1,h)(g,\phi(g))(1,h)^{-1} = (g',\phi(g'))$ for some $g' \in G_1$.
Then
\begin{align*}
(g',\phi(g')) 	& = (1,h)(g,\phi(g))(1,h)^{-1}\\
			 	& = (1,h)(g,\phi(g))(1,h^{-1})\\
				& = (g,h\phi(g)h^{-1}).
\end{align*}
So, $g'=g$.
Hence, for all $h \in \widehat{G}/\widehat{K}$ and $g \in G_1$, we have $\phi(g) = h\phi(g) h^{-1}$. So, $\phi(G_1)$ is in the center of $\widehat{G}/\widehat{K}$.
Since $\phi(G_1) = \widehat{G}/\widehat{K}$, this group is abelian, which contradicts our assumption that $G_1$ is non-abelian. 
\end{proof}

In particular, $G=G_1\times\cdots\times G_n$ has the remarkable property that if $N$ is a normal subgroup of $G$ and $M$ is a normal subgroup of $N$, then $M$ is normal in $G$.  

We now return to our shuffle groups.
It is clear from the definitions that $\Gamma_n\subset \Gamma_{n+1}$ and we have
an ascending chain of finite products of symmetric groups. 
Let $A_{n,a} \subset \Gamma_{n,a}$ be the alternating subgroup.
That is, $A_{n,a}$ is the subgroup $\{f_\alpha \mid \alpha \in \Sym(\rsub^n(a)), \alpha \mbox{ is even}\}$.
Let $A_n = \prod A_{n,a}$.
By construction, $A_n$ is a product of simple groups.
Finally, we define $A = \bigcup_{n \geq 1} A_n$ and call this subgroup the \emph{alternating shuffle group}.
While the group $H$ in Example \ref{EX:thompson} is not an alternating shuffle group, it is a subgroup of $A$.
Our goal will be to generalise that construction and to find an analogue for the group $H$ in general, which is an infinite simple subgroup of shuffles.

\begin{lemma}\label{LEM:chain}
Let $d$ be the cardinality of the alphabet $\mathcal{A}$.
The groups $A_n$ form an ascending chain and $A\subset\Gamma$ is a normal subgroup of index 
bounded above by $2^d$. 
\end{lemma} 
\begin{proof}
Consider the composition $q\circ i$ of the inclusion $i\colon\Gamma_n\to \Gamma_{n+1}$ and the quotient map $q\colon\Gamma_{n+1} \to \Gamma_{n+1}/A_{n+1}\cong (\Z/2\Z)^d$.
It maps $\Gamma_n$ to an abelian group.
Since $A_n$ is the commutator subgroup of $\Gamma_n$, it is in the kernel of $q\circ i$.
In other words, $A_n$ is a subgroup of $A_{n+1}$.
We have $\Gamma/A=\bigcup \Gamma_n A/A$, an increasing union of sets.
Since $\Gamma_n A/A\cong \Gamma_n/(\Gamma_n\cap A)$ and since $\Gamma_n/(\Gamma_n\cap A)$ is a quotient of $\Gamma_n/A_n$, its order divides $2^{d}$.
Therefore, $[\Gamma : A]$ divides $2^{d}$. 
\end{proof}

We note that in our considerations, how $\Gamma_n$ embeds into $\Gamma_{n+1}$ is unimportant, only that it is an embedding.
It is possible to define a chain $G_1\subset G_2\subset\cdots$ of products of symmetric groups such that each $G_n$ embeds in the alternating group of the first factor of $G_{n+1}$.
In such a case, $G = \bigcup G_n$ would be a union of an ascending chain of symmetric groups, which is equal to the union of the alternating subgroups $A$.
In other words, the index of $A$ would be equal to 1.

\begin{coro}
All non-trivial normal subgroups of the alternating shuffle group $A$ have infinite index.
Hence, $A$ is not residually finite.
\end{coro}

\begin{proof}
Suppose $N\subseteq A$ is a non-trivial normal subgroup.
Then each $N\cap A_n$ is normal in $A_n$.
For large enough $n$, $N \cap A_n$ is eventually always non-trivial, as otherwise, every $N \cap A_{n}$ would be trivial (because $N \cap A_n \subseteq N \cap A_{n+1}$), contradicting $N \cap A$ being non-trivial.
Without loss of generality, assume that $A_{n,a}$ is simple for all $n \geq 1$ and $a \in \mc A$ (if not, by branching, replace $\rsub$ with a large enough power $\rsub^p$ so that $\#\rsub^p(a) \geq 5$ for all $a$).
Then, since $A_n$ is a product of simple groups, either $N \cap A_n$ is equal to $A_n$, or it has index at least as large as the smallest cardinality of one of its factors $A_{n,a}$ because $N \cap A_n$ is eventually non-trivial.
Since these cardinalities increase to $\infty$ with $n$, either $N$ is equal to $A$ or it has infinite index.
For a group to be residually finite, it must contain at least one non-trivial normal subgroup of finite index.
\end{proof}

\begin{theorem}\label{THM:alternating-shuffle-simple}
The alternating shuffle group $A$ contains an infinite simple subgroup.
\end{theorem}

\begin{proof}
First of all observe that $A$ inherits the remarkable property that if $N$ is normal in $A$ and $M$ is normal in $N$, then $M$ is normal in $A$.
By Lemma \ref{LEM:factor}, every descending chain of normal subgroups in $A_n$ has length at most $d+1$.
This implies that every descending chain of normal subgroups in $A$ has length at most $d+1$.
Let $N$ be the last non-trivial normal subgroup in such a chain.
It contains no proper normal subgroup, since every $M$ that is normal in $N$ is normal in $A$.
$N$ is a simple group.  

That $N$ is infinite follows from the fact that $\rsub$ is branching and so $\#\rsub^n(a) \to \infty$ as $n \to \infty$ for each $a \in \mc A$. Therefore, for each $a$, $\#A_{n,a} \to \infty$.
\end{proof}

Hence, the shuffle group $\Gamma$ and the full automorphism group $\Aut(X_\rsub)$ contain an infinite simple subgroup and are not residually finite.

While the shuffle group is a large subgroup of the automorphism group, there are many more varied automorphisms for RS-subshifts.
By generalising the shuffle procedure slightly; allowing shuffles to depend on local information (in analogue to finite block codes), we get a more complete picture of how the automorphism group looks.
This is what we explore in the next section.

\section{Automorphism groups for recognisable substitutions}\label{SEC:recog}
\begin{definition}
We say that the letter $a \in \mc A$ \emph{appears with uniformly bounded gaps} in a subshift $X$ if the set $\{aua \triangleleft x, x \in X \mid u \in (\mc A\setminus\{a\})^\ast\}$ is finite.
\end{definition}
Equivalently, the letter $a \in \mc A$ appears with uniformly bounded gaps in $X$ if and only if there exists a uniform bound $m$ such that if $|u|=m$ and $u \triangleleft x$ for $x \in X$ then $|u|_a \geq 1$.
\begin{lemma}\label{LEM:uni-bound-gaps}
Let $\rsub$ be a primitive, compatible random substitution.
For every $a$ in $\mc A$, the letter $a$ appears infinitely often in every element in $X_\rsub$ and with uniformly bounded gaps.
\end{lemma}
\begin{proof}
By primitivity, we may suppose without loss of generality that every entry of $M_\rsub$ is positive.
Let $ {\displaystyle m = \max_{a \in \mc A}\{|\rsub(a)|\}}$.
Let $a\in \mc A$ be some fixed letter and let $x \in X_\rsub$ be an arbitrary element in the RS-subshift.
By results in \cite{RS:random}, there exists an element $y \in X_\rsub$ such that $\sigma^{-i}(x) \in \rsub(y)$ for some $ 0 \leq i \leq |\rsub(y_0)|$.
It follows that $x$ can be decomposed into an exact concatenation of words of the form $\rsub(y_j)$.
As $M_\rsub$ is positive, every such word contains the letter $a$, and so $x$ contains infinitely many appearances of the letter $a$ and each appearance of $a$ is at most a distance $2m$ apart.
As $m$ depends only on $M_\rsub$, it is a uniform bound for all $x$, as required.
\end{proof}
\begin{remark}
We should mention that the above Lemma is perfectly compatible with the fact that $X_\rsub$ is, in general, not uniformly recurrent.
All that Lemma \ref{LEM:uni-bound-gaps} gives us is uniformly bounded returns to single letters, rather than words.
In fact, for suitably nice RS-subshifts, there must exist a word $u$ that has arbitrarily large gaps between occurrences in some elements of the subshift.
This will be proved later and used when studying the constant length case in Section \ref{SEC:const-length}.
\end{remark}


\begin{theorem}\label{THM:auto-full-shift}
Let $\rsub$ be a non-degenerate, compatible, primitive, recognisable random substitution.
Let $X_n$ be the full shift on $n$ letters.
The automorphism group $\Aut(X_{n})$ embeds into the automorphism group $\Aut(X_\rsub)$.
\end{theorem}
\begin{proof}
It is well-known that for every $n,m \geq 2$ the automorphism group $\Aut(X_n)$ embeds into the automorphism group $\Aut(X_m)$ \cite{KR:automorphisms}, and so it is enough to show that for at least one $n$, $\Aut(X_n)$ embeds into $X_\rsub$.
Let $a \in \mc A$ be any letter in the alphabet.
Without loss of generality, by Proposition \ref{PROP:branching}, suppose that $\#\rsub(a) \geq 2$ and let $n = \#\rsub(a)$ (if not, take a higher power $\rsub^k$, for which all of the assumptions on $\rsub$ still hold).
Write $\rsub(a) = \{u_{1} \ldots, u_{n}\}$.
For every $x \in X_\rsub$, by recognisability, we may write $x$ uniquely as a concatenation of words
\[
\cdots \square \rsub(a) \square \rsub(a) \square \rsub(a) \square \rsub(a) \square \rsub(a) \square \rsub(a) \square \cdots,
\]
where $\square$ is used to represent some collection of inflation words not of type $a$, that is, $\square$ is a word of the form $\rsub(w)$ with $w \in \mc (\mc A\setminus\{a\})^\ast$.
Our arguments will all be compatible with the action under the shift, and so it is not necessary for us to identify the origin of $x$.
Note that the possible realisations of the $\square$s are uniformly bounded in length by Lemma \ref{LEM:uni-bound-gaps}.
That is, there exists some $M \geq 1$ such that for all realisations of $\square$, $|\square| \leq M$.

So, there exists a unique bi-infinite sequence $ \omega \in \{1, \ldots, n\}^\Z$  (up to a shift) such that
\[
x = \cdots \square u_{\omega(-2)}\square u_{\omega(-1)} \square u_{\omega(0)} \square u_{\omega(1)} \square u_{\omega(2)} \square \cdots.
\]
Let $\alpha$ be an automorphism of $X_{n}$.
We define a map $f_\alpha \colon X_\rsub \to X_\rsub$ by 
\[
f_\alpha(x) = \cdots \square u_{(\alpha\omega)(-2)}\square u_{(\alpha\omega)(-1)} \square u_{(\alpha\omega)(0)} \square u_{(\alpha\omega)(1)} \square u_{(\alpha)\omega(2)} \square \cdots.
\]
To see that $f_\alpha(x) \in X_\rsub$, note that, by recognisability, there exists a unique $y \in X_\rsub$ and a unique $0 \leq k \leq |\rsub(y_0)|-1$ such that $\sigma^{-k}(x) \in \rsub(y)$.
This unique preimage $y$ is also a preimage of $f_\alpha(x)$, by construction of $f_\alpha(x)$, as the placement of inflation words and their types has not changed, only their realisations.
That is, $\sigma^{-k}(f_\alpha(x)) \in \rsub(y)$.
As $\rsub(y) \subset X_\rsub$, it follows that $f_\alpha(x) \in X_\rsub$.

The map $f_\alpha$ is an example of a generalised shuffle as described in Remark \ref{REM:gen-shuff}, whose radius is uniformly bounded in length.
This follows from the fact that $\alpha$ is a sliding block code, all realisations of $\square$ are uniformly bounded in length, and the length $|\rsub(a)|$ is fixed for all realisations.
As such $f_\alpha$ is itself a sliding block code (again, due to local recognisability), and so is a shift-commuting continuous function on $X_\rsub$.

It is not difficult to see that $f_\alpha \circ f_{\alpha^{-1}} = f_{\alpha^{-1}} \circ f_\alpha = \operatorname{Id}_{X_\rsub}$, and so $f_\alpha \in \Aut(X_\rsub)$.
Further, it is clear that $f_{\alpha \circ \beta} = f_\alpha \circ f_\beta$ and so the map $\Aut(X_{n}) \to \Aut(X_\rsub) \colon \alpha \mapsto f_\alpha$ is a group homomorphism.

As $f_\alpha(x)$ is derived from $x$ by replacing inflation words with other inflation words of the same type, $f_{\alpha}(x) = x$ if and only if $u_{\alpha\omega(i)} = u_{\omega(i)}$ for every $i$, hence $\alpha\omega = \omega$.
If $\alpha$ is non-trivial, then there exists some $\hat{\omega} \in \{1, \ldots, n\}^\Z$ such that $\alpha\hat{\omega} \neq \hat{\omega}$, and so
\[
f_\alpha ( \cdots \square u_{\hat{\omega}(-1)} \square u_{\hat{\omega}(0)} \square u_{\hat{\omega}(1)} \square \cdots) \neq  \cdots \square u_{\hat{\omega}(-1)} \square u_{\hat{\omega}(0)} \square u_{\hat{\omega}(1)} \square \cdots.
\]
It follows that $f_\alpha$ is not the identity, and so the map $\Aut(X_{n}) \to \Aut(X_\rsub) \colon \alpha \mapsto f_\alpha$ is an embedding.
\end{proof}
As an immediate corollary, any subgroup of $\Aut(X_{2})$ embeds into $\Aut(X_\rsub)$, and so by results of Boyle, Lind and Rudolph \cite{BLR:auto-sft}, we can also identify the following subgroups of $\Aut(X_\rsub)$.
\begin{coro}
Let $\rsub$ be a non-degenerate, compatible, primitive, recognisable random substitution.
The following groups embed into the automorphism group $\Aut(X_\rsub)$:
\begin{itemize}
\item $\Aut(X_n)$, $n \geq 2$,
\item every countable direct sum of finite groups,
\item Any finite free product $(\Z/2\Z)^{\ast n} = \Z/2\Z \ast \cdots \ast \Z/2\Z$, hence the free group $F_2$ on $2$ generators, hence the free group $F_n$ on $n \geq 1$ generators, and the free group $F_\omega$ on countably many generators,
\item the countable direct sum ${\displaystyle \Z^\omega = \bigoplus_{i=1}^{\infty} \Z}$.
\end{itemize}
Moreover, $\Aut(X_\rsub)$ contains all of the other subgroups mentioned in the discussion in Section \ref{SUBSEC:autos}.
\end{coro}
Given the above, a natural question is whether there are subgroups of $\Aut(X_\rsub)$ which do not appear as subgroups of the group $\Aut(X_2)$.
In fact, this follows from Theorem \ref{THM:alternating-shuffle-simple}, as the alternating shuffle group $A$ contains an infinite simple group. $\Aut(X_2)$ is residually finite (by Proposition \ref{PROP:periodic}), and so $A$ cannot appear as a subgroup of $\Aut(X_2)$.
It follows that $\Aut(X_\rsub)$ is in a very concrete sense `larger' than $\Aut(X_2)$.

\section{Local recognisability}\label{SEC:local-recog}
Gohlke, Rust and Spindeler introduced the notion of an \emph{inflation word decomposition} for random substitutions \cite{GRS:random-sft}.
This definition is key to defining a local version of recognisability, as is done for deterministic substitutions.
This is also a useful framework in which to define what it means for a single word in the language to be recognisable.
This notion will be used in Section \ref{SEC:const-length} to relax recognisability.

\begin{definition}
Let $\vartheta$ be a random substitution and let $u\in \mathcal{L}_{\vartheta}$ be a legal word.
Let $n\geq 1$ be a natural number. 
For words $u_i \in \mathcal{A}^+$, the tuple $[u_1,\ldots,u_{\ell}]$ is called a \emph{$\vartheta$-cutting} of $u$ if $u_1\cdots u_{\ell}=u$ and there exists a $\vartheta$-legal word $v=v_1\cdots v_{\ell}$ such that 
\begin{itemize}
    \item For $i=2,\ldots, \ell-1$, $u_i$ is an inflation word built from the letter $v_i$.
    That is, $u_i \in \vartheta(v_i)$. Note that $u_i$ is a word, while $v_i$ is a single letter. 
    \item $u_1$ is the suffix of an inflation word built from $v_1$, and 
    \item $u_\ell$ is the prefix of an inflation word built from $v_\ell$. 
\end{itemize}
That is, $u$ is contained in a realisation of $\vartheta(v)$, which is a concatenation of 
inflation words, with each of the interior $u_i$'s being one of those inflation words. 

We call $v$ a \textit{root} of the $\vartheta$-cutting and we call $\left([u_1,\ldots,u_{\ell}],v\right)$ the corresponding \emph{inflation word decomposition} of $u$. If 
$u$ actually is a realisation of $\vartheta(v)$ for some $v$, so $u_1\in \vartheta(v_1)$ and 
$u_\ell \in \vartheta(v_\ell)$, then we say that $u$ is an {\em exact} inflation word.
Finally, we let $D_{\vartheta}(u)$ denote the set of all decompositions of the word $u$.
\end{definition}

\begin{example}
Let $\vartheta \colon a\mapsto \{ab, ba\}, b\mapsto \{aa\}$ be the random period doubling substitution.
The word $aab$ has two possible $\vartheta$-cuttings $[a,ab]$ and $[aa,b]$ with two and one distinct roots respectively.
The set of inflation word decompositions of $aab$ is 
\[
D_\vartheta(aab)=\left\{\left([a,ab], aa\right), \left([a,ab], ba\right), \left([aa,b],ba\right)\right\}.
\]
For this random substitution, $abba$ is an exact inflation word.
However, $bb$ is not an exact inflation word, since any concatenation of inflation words containing $bb$ must also contain some $a$s.
\end{example}

Let $u$ be a $\vartheta$-legal word and let $u_{[i,j]}$
be a subword. An inflation word decomposition of
$u$ restricts to an inflation word decomposition
of $u_{[i,j]}$, which we call an {\em induced inflation word decomposition}. The idea is simple, but the precise definition is somewhat
technical: 
\begin{definition}
Let $\vartheta$ be a random substitution and let $d = ([u_1, \ldots, u_\ell], v)\in D_{\vartheta}(u)$ be an inflation word decomposition of $u$. 
For $0\leq i \leq j \leq |u|-1$, we write $d_{[i,j]}$ for the \emph{induced inflation word decomposition on the subword} $u_{[i,j]}$, defined by
$$d_{[i,j]}=\left([u_1, \ldots, u_\ell], v\right)_{[i,j]}=\left([\hat{u}_{k(i)},u_{k(i)+1},\ldots,u_{k(j)-1},\hat{u}_{k(j)}], v_{[k(i),k(j)]}\right),$$ 
where $1 \leq k(i) \leq k(j) \leq \ell$ are natural numbers such that
$$\left|u_1\cdots u_{k(i)-1}\right| < i+1 \leq \left|u_1\cdots u_{k(i)}\right| \text{ and } \left|u_1 \cdots u_{k(j)-1}\right|< j+1 \leq \left|u_1 \cdots u_{k(j)}\right|,$$
$\hat{u}_{k(i)}$ is a suffix of $u_{k(i)}$ and $\hat{u}_{k(j)}$ is a prefix of $u_{k(j)}$ such that
$$\hat{u}_{k(i)}u_{k(i)+1}\cdots u_{k(j)-1}\hat{u}_{k(j)}=u_{[i,j]}.$$

\end{definition}

\begin{example}\label{eg:induceddecom}
Let $\vartheta \colon a\mapsto \{ab, ba\}, b\mapsto \{aa\}$ be the random period doubling substitution.
The legal word $u=ababa \in \mathcal{L}_\vartheta$ has exactly four inflation word decompositions given by 
\[
D_\vartheta(u) = \big\{ ([a,ba,ba],aaa), ([a,ba,ba],baa), ([ab,ab,a],aaa), ([ab,ab,a],aab)\big\}.
\]
For the subword $u_{[1,3]}=bab$ of $u$, the first two elements of $D_\vartheta(u)$ yield the induced decomposition 
$d^{(1)}_{[1,3]}=([ba,b],aa)$, while the next two elements yield $d^{(2)}_{[1,3]}=([b,ab],aa)$. 
Thus, $\#\left\{d_{[1,3]} \mid d\in D_\vartheta(u)\right\} = 2$. 
In this example, both possible $\vartheta$-cuttings of $bab$ are induced from cuttings of $u$. 

By contrast, the word $u' = bbaba$ has a unique inflation word decomposition $d' = \left( [b,ba,ba], aaa \right)$ which yields a unique induced inflation word decomposition on the subword $u'_{[1,3]} = bab$ given by $d'_{[1,3]} = \left([ba,b],aa \right)$.
That is, when $bab$ sits inside $u$, the embedding tells us nothing about the $\vartheta$-cutting of $bab$, but when $bab$ sits inside $u'$, the embedding uniquely defines the $\vartheta$-cutting of $bab$.
\end{example}

\begin{definition}
Let $\vartheta$ be a random substitution and let $u\in \mathcal{L}_\vartheta$ be a legal word. 
We say that $u$ is \emph{recognisable} if there exists a natural number $N$ such that for each legal word of the form $w=u^{(l)}uu^{(r)}$ with $\left|u^{(l)}\right|=\left|u^{(r)}\right|=N$, all inflation word decompositions of $w$ induce the same inflation word decomposition of $u$. That is, knowing
the $N$ letters to the left of $u$ and the $N$ letters
to the right of $u$ determines a unique induced inflation
word decomposition of $u$. 
We call the minimum such $N$ the \emph{radius of recognisability} for $u$.
\end{definition}

\begin{definition}
We call a random substitution $\vartheta$ \emph{locally recognisable} if there exists a natural number $N$ such that every $\vartheta$-legal word is recognisable with radius at most $N$.  
The minimum such $N$ is called the \textit{radius of recognisability for $\vartheta$}.
\end{definition}

Intuitively, this means that there is a `window' of radius $2N+1$ such that if we place the window somewhere in a word, the induced inflation word decomposition of the central letter in the window is uniquely determined.
Then, for a long enough word $u$, as we slide the window along the word, all but the very extremal inflation words that make up $u$ are determined.
That is, we know how to uniquely decompose $u$ into inflation words except possibly at the very ends of the word.
Without loss of rigour, we will normally work with this intuitive idea of local recognisability rather than the explicit definition, as the notation can become cumbersome. 

Equivalence of (global) recognisability and local recognisability is the next result.
\begin{prop}\label{PROP:local-global}
A compatible random substitution $\rsub$ is recognisable if and only if it is locally recognisable.
\end{prop}
\begin{proof}
($\implies$)
We prove the contrapositive.
Suppose that $\rsub$ is not locally recognisable.
Then for every $N \geq 0$, there exists a word $v_N = v$ of the form $v = u^{(l)}uu^{(r)}$ such that $|u^{(l)}| = |u^{(r)}| = N$ and such that the induced inflation word decomposition of $u$ is not unique.
Note that we can assume $|u|=2$ without loss of generality.
Extend $v_N$ arbitrarily to a bi-infinite element $x_N \in X_\rsub$.
By compactness, the sequence $(x_N)_{N \geq 0}$ has an accumulation point $x \in X_\rsub$.
Let $n \in \N$.
We have $x_{[-n,n]} = (x_N)_{[-n,n]}$ for some $N \in \N$ and hence there are two inflation word decompositions of $x_{[-n,n]}$ that differ at the origin.
Sending $n \to \infty$ and using compactness for the $\rsub$-preimages, we find two distinct preimages of $x$ in $X_\rsub$.
It follows that $\rsub$ is not recognisable.

($\impliedby$)
Let $x \in X_\rsub$ be an arbitrary element.
Let $N$ be the radius of recognisability for $\rsub$.
Existence of a preimage $y$ for $x$ under $\rsub$ is already known from work of Rust and Spindeler \cite{RS:random}.
That is, there exists an element $y \in X_\rsub$ and $0 \leq k \leq |\rsub(y_0)|-1$ such that $\sigma^{-k}(x) \in \rsub(y)$.
It remains to show uniqueness.
Let $n \geq 0$ be an arbitrary natural number and consider the word $u \coloneqq x_{[-n,n]}$.
We note that $u$ is a subword of the word $u' = x_{[-n-N,n+N]}$ and by local recognisability, all inflation word decompositions of $u'$ must induce the same inflation word decomposition on $u$.
So any preimage $y$ must coincide with the root of this inflation word decomposition for $u$ on some central subword  $v = y_{[-n',n']}$, where $n'$ is approximately $n' \approx n/|\rsub|$ up to a constant, where $|\rsub| = \max\{|\rsub(a)| \mid a \in \mc A\}$.
As $n$ was arbitrary and $\frac{n}{|\rsub|} \to \infty$, it follows that arbitrarily large central subwords of $y$ are determined by $x$.
Hence, $y$ is unique.
As $x$ was arbitrary, $\rsub$ is recognisable.
\end{proof}
The definition of local recognisability used here was a delicate notion to formalise, as the choice of images of letters under random substitutions means that one must be more careful than in the deterministic setting in order to capture the correct meaning (which is relatively simple as far as intuition is concerned), especially as one really wants the definition to be equivalent to global recognisability.
Both the definition of local recognisability and the above proof were pieced together from those appearing in \cite{F:book} in the context of deterministic substitutions and conversations with Philipp Gohlke, to whom we are extremely grateful.

\section{Automorphism groups for constant length random substitutions}\label{SEC:const-length}
Recognisability is used heavily in Section \ref{SEC:recog} in order to ensure that the automorphisms being constructed are well-defined.
In general, we do not have recognisability and so shuffles and generalised shuffles are no longer well-defined.
This includes important examples for which one is interested in studying the associated automorphism groups, such as the random Fibonacci \cite{GL:random,N:fibonacci-entropy} and random period doubling substitutions \cite{BSS:rand-diffraction}.
The random period doubling substitution is especially of interest as the associated RS-subshift contains periodic points\footnote{This was first observed by Rust and Spindeler \cite{RS:random} and investigated in more detail by Rust \cite{R:random-periodic}.}, unlike for RS-subshifts associated with recognisable substitutions.
Rust and Spindeler showed that primitive random substitutions that contain a periodic point actually contain a dense set of periodic points \cite{RS:random}, so by Proposition \ref{PROP:periodic}, the automorphism group of the random period doubling subshift is residually finite.
The existence of periodic points in the RS-subshift therefore forces the automorphism group to be `smaller' in this sense (compare Theorem \ref{THM:alternating-shuffle-simple}).
In general, we have the following.
\begin{coro}\label{COR:periodic-res-finite}
Let $\rsub$ be a primitive random substitution.
If $X_\rsub$ contains a single periodic element, then $\Aut(X_\rsub)$ is residually finite.
\end{coro}
For a discussion of the existence/non-existence and enumeration of periodic points for compatible random substitutions, the reader is directed to the work of Rust \cite{R:random-periodic}.

In order to study properties such as subgroups for automorphism groups of RS-subshifts without recognisability, it is helpful to introduce other restrictions on our random substitutions.
Our aim is to introduce conditions that are satisfied by the random period doubling substitution, so that consequent results apply to this example.

To motivate this direction, let us construct a non-trivial automorphism of the random period doubling substitution.
\begin{example}
Recall that the random period doubling substitution is given on the alphabet $\mc A = \{a,b\}$ by $\rsub \colon a \mapsto \{ab,ba\}, b \mapsto \{aa\}$.
This is a primitive, compatible random substitution of constant length $L=2$, which is not recognisable, but for which the word $bb$ is recognisable with radius of recognisability $0$, admitting the unique inflation word decomposition $([b,b],(aa))$ (see Example \ref{EX:period-doubling-recog}).

We can use the word $bb$ as a marker between which inflation words may be replaced with other inflation words of the same type.
For instance, the following is the list of all words of length $8$ in the language $\mc L_\rsub$ that both begin and end with $bb$:
\[
u_1 = bbaaaabb,\:\: u_2 = bbaababb,\:\: u_3 = bbabaabb
\]
Because $bb$ has a unique inflation word decomposition, it actually means that these three words $u_i$, $i=1,2,3$ also have unique inflation word decompositions.
Notice that $u_2$ and $u_3$ have the same root.
That is $u_2$ and $u_3$ are uniquely decomposed as 
\[
u_2 = \underbracket{-b}_{a}\underbracket{ba}_{a}\underbracket{ab}_{a}\underbracket{ab}_{a}\underbracket{b-}_{a},\:\: u_3 = \underbracket{-b}_{a}\underbracket{ba}_{a}\underbracket{ba}_{a}\underbracket{ab}_{a}\underbracket{b-}_{a}.
\]
We may therefore define an automorphism of $X_\rsub$ in the following way.
The automorphism $f \colon X_\rsub \to X_\rsub$ replaces any occurrence of the word $u_2$ with the word $u_3$ and vice versa.
This is well-defined because if $u_2$ and $u_3$ were ever to overlap in an element of the subshift\footnote{In fact, this cannot happen because $a^9$ is not legal for $\rsub$.}, then the rule defining $f$ can still act on both words, as the place where they overlap, the word $bb$, is fixed by the rule.
In this sense, $bb$ is a \emph{marker}.
As $u_2$ and $u_3$ have the same unique root, $f$ is well-defined on $X_\rsub$.
By now, it is also clear that $f$ really is an automorphism and is its own inverse.
There are plenty of elements in the subshift for which $u_2$ and $u_3$ appear as subwords, and so it is clear that $f$ is non-trivial.
We hence see that $f$ is a non-trivial involution in $\Aut(X_\rsub)$.

Notice that it was necessary to define our automorphism in such a way that a nearby word ($bb$) was recognisable, in order to make the mapping well-defined.
If we would try to define an automorphism that, for example, just replaced the word $ab$ with $ba$ and vice versa, this would not be well-defined on all $x \in X_\rsub$ due to the lack of recognisability for $\rsub$ (and several other reasons).

The longer we make our words $u_i$ (which contain $bb$ on the ends but not in between), the more choice we have for applying similar inflation word replacement, such as the automorphism that performs the cyclic replacement
\begin{center}
\begin{tikzcd}
bba ab aa ab ab b \arrow[r] & bba ba aa ab ab b \arrow[d]\\
bba ab aa ab ab b \arrow[u] & bba ba aa ab ab b \arrow[l]
\end{tikzcd}
\end{center}
on words that share the unique root $aaabaaa$.
One just has to ensure that no new occurrences of $bb$ are being created between the two `border' $bb$s.
\end{example}
The above mentioned properties will form the basis for a general argument and so we now take the time to formalise them, in order to understand better how to apply them in a general setting.

\subsection{Relaxing recognisability}
If there exists a recognisable word $u \in \mc L_\rsub$, then we say that $\rsub$ \emph{admits a recognisable word}.
We say that the random substitution $\rsub$ \emph{has disjoint images of letters} if, for all $a, b \in \mc A$, $\rsub(a) \cap \rsub(b) \neq \varnothing \implies a = b$.
We say that the random substitution $\rsub$ \emph{has disjoint images} if, for all legal words $u, v \in \mc L_\rsub$, $\rsub(u) \cap \rsub(v) \neq \varnothing \implies u = v$.
Although recognisability neither implies nor is implied by the disjoint images property, they are intuitively similar, as they both restrict how legal words may be `desubstituted'.
On the other hand, the existence of a recognisable word is a direct relaxation of recognisability.

A Recognisable word may always be extended to a minimal recognisable word that constitutes an exact inflation word.
This is done by picking a long enough superword $w \in \mc L_\rsub$ of $u$ that induces a unique inflation word decomposition $([u_{(0)}, \ldots, u_{(\ell)}],v) \in D_\rsub(u)$ and extending $u$ within $w$ to $u' = u'_{(0)} u_{(1)} \cdots u_{(\ell - 1)} u'_{(\ell)} \in \rsub(v)$, which is also recognisable by construction.
We may therefore always assume that a recognisable word is an exact inflation word when it is convenient to do so and that each of the decomposition words $u_{(0)}, \ldots, u_{(\ell)}$  for the recognisable word $u$ are images of the random substitution on a single letter, so $u \in \vartheta(v)$ for some $v \in \mc L_\vartheta$.

It has previously been shown that a constant length random substitution has disjoint images of letters if and only if it has disjoint images \cite{R:random-periodic}.
In particular, a constant length random substitution $\rsub$ has disjoint images of letters if and only if all of its powers $\rsub^p$ also have disjoint images of letters.
In the case of non-constant length substitutions, disjoint images can be strictly stronger than disjoint images of letters, as demonstrated by the random Fibonacci substitution $\rsub \colon a \mapsto \{ab,ba\}, b \mapsto \{a\}$, which has disjoint images of letters.
Nevertheless, we see that $\rsub$ does not have disjoint images because $\rsub(ab) \cap \rsub(ba) = \{aba\}$.

The following result says that the existence of a recognisable word in an element $x$ for a constant length random substitution enforces unique cutting points (boundaries of inflation words) in the whole of $x$.
This result will be used later to ensure that an automorphism is well-defined.
Notice that $y$ is not necessarily unique in the following result.
\begin{lemma}\label{LEM:const-length-recog}
Let $\rsub$ be a substitution of constant length $L$ which admits one recognisable word $u$.
If $x \in X_\rsub$ contains at least one appearance of the word $u$ as a subword, then there exists a unique $0 \leq i \leq L-1$ such that $\sigma^{-i}(x) \in \rsub(y)$ for some $y \in X_\rsub$.
\end{lemma}
\begin{proof}
Without loss of generality, assume that $u \in \vartheta(\hat{v})$ for some $\hat{v} \in \mc L_\vartheta$.
Suppose that the radius of recognisability for $u$ is $N$.
Let $i\leq j$ be such that $x_{[i,j]} = u$.
Suppose that the unique inflation word decomposition of $u$ induced by any inflation word decomposition of $x_{[i-N,j+N]}$ is given by $([u_{(0)}, \ldots, u_{(\ell)}], v)$.
By shifting $x$, we may assume without loss of generality that $x_{[0, |u|-1]} = u$.

We know that $x$ has at least one preimage $y$ under the substitution.
As $\rsub$ is constant length, say with length $L$, and the subword $x_{[0, |u|-1]}$ has a unique decomposition into inflation words which is compatible with an inflation word decomposition of the whole of $x$, it means that $x \in \rsub(y)$, with no allowed shifting of $x$.
Or more succinctly, $i = 0$, and so is unique.
\end{proof}
If we further assume that $\rsub$ has disjoint images of letters, then not only is $i$ unique, but so is the preimage $y$.
So not only do we enforce the cutting point structure in $x$, but also the types of the inflation words.
\begin{prop}\label{PROP:const-length-recog}
Let $\rsub$ be a substitution of constant length $L$ with disjoint images of letters and which admits one recognisable word $u$.
If $x \in X_\rsub$ contains at least one appearance of the word $u$ as a subword, then there exists a unique $y \in X_\rsub$ and $0 \leq i \leq L-1$ such that $\sigma^{-i}(x) \in \rsub(y)$.
\end{prop}
\begin{proof}
That $i$ is unique follows from Lemma \ref{LEM:const-length-recog}, and so without loss of generality, assume that $i=0$.
In particular, every subword of $x$ of the form $x_{[iL, (i+1)L-1]}$ is an exact inflation of a single letter.
That is, $x_{[iL, (i+1)L-1]} \in \rsub(y_i)$, where $y$ is some preimage of $x$ under $\rsub$.
But, as $\rsub$ has disjoint images of letters, we can conclude that the word $x_{[iL, (i+1)L-1]}$ uniquely determines the letter $y_i$.
It follows that $y$ is uniquely determined by $x$, and so both the element $y$ and the integer $i$ are unique.
\end{proof}
Hence, for constant length substitutions with disjoint images of letters, a suitable abundance of recognisable words actually gives recognisability of the random substitution.
\begin{coro}\label{CORO:const-length-recog}
Let $\rsub$ be a substitution of constant length with disjoint images of letters and suppose that every element in the RS-subshift $X_\rsub$ contains at least one recognisable subword.
Then $\rsub$ is recognisable.
\end{coro}
\begin{example}\label{EX:period-doubling-recog}
Let $\rsub \colon a \mapsto \{ab, ba\}, b \mapsto \{aa\}$ be the random period doubling substitution.
It is a constant length random substitution with length $L = 2$.
The word $aa$ is not recognisable for $\rsub$ but showing this is not so simple.
Note that $aa$ has five inflation word decompositions
\begin{eqnarray*}
D_\rsub(aa) & = & \{([aa],b), ([a,a],bb),\\
& &([a,a],ba), ([a,a],ab), ([a,a],aa)\}
\end{eqnarray*}
and for any $n \geq 0$, the word $u_n = (aabaab)^n aa (baabaa)^n$ is legal for $\vartheta$.
We then notice that $u_n$ admits the inflation word decompositions
\[
([\overbrace{aa, ba, ab}^{\times 2n}, aa], (baa)^{2n}b ) \quad \text{ and } \quad ([a, \overbrace{ab, aa, ba}^{\times 2n}, a], a(aba)^{2n}a ),
\]
which induce different inflation word decompositions on the central word $aa$ (even different $\vartheta$-cuttings).
In particular, they induce the inflation word decompositions $([aa], b)$ and $([a,a], aa)$ respectively. Hence, no radius of recognisability exists for $aa$.

On the other hand, the word $bb$ is recognisable, as the only inflation word decomposition for $bb$ (regardless of how it is induced) is
\[
D_\rsub(bb) = \{([b,b],aa)\}.
\]
Although $\rsub$ admits a recognisable word $bb$ (and in fact infinitely many recognisable words), it is not true that $\rsub$ is recognisable, as $bb$ does not appear as a subword of every element in $X_\rsub$, as demonstrated by the element $(aab)^\infty \in X_\rsub$.
In fact, by Corollary \ref{CORO:const-length-recog}, not every element of $X_\rsub$ can contain a recognisable word as a subword, since the existence of a periodic element in the subshift rules out recognisability of $\rsub$ by results from \cite{R:random-periodic}.
\end{example}

We will need to also control the existence of recognisable words for powers of substitutions.
\begin{lemma}\label{LEM:recog-word}
Let $\rsub$ be a primitive, constant length random substitution with disjoint images of letters and let $p \geq 1$ be an integer.
The random substitution $\rsub$ admits a recognisable word if and only if $\rsub^p$ admits a recognisable word.
\end{lemma}
\begin{proof}
($\implies$)
Write $L$ for the length of the substitution, so for all $a \in \mc A$, $|\vartheta(a)| = L$.
Let $u$ be a recognisable word for $\rsub$ with radius of recognisability $N$ and let $\hat{u} = u^{(l)}uu^{(r)}$ be a legal of extension of $u$ with $|u^{(l)}| = |u^{(r)}| = N$.
Suppose that all inflation word decompositions of $\hat{u}$ induce the unique inflation word decomposition $([u_{(0)}, \ldots, u_{(\ell)}], v)$ on $u$ and without loss of generality assume that $u \in \rsub(v)$.
By primitivity, there exists a legal word $w \in \mc L_\rsub$ that contains both $\hat{u}$ and $v$.
Consider now a word $w' \in \rsub(w)$ which is given by realising the subword $\rsub(v)$ to be $u$.
As $u \triangleleft w'$, $u$ is recognisable for $\rsub$ and $\rsub$ is constant length, this forces the possible $\rsub$-cuttings of $w'$ to be fixed as soon as we know an $N$-neighbourhood of $w'$.
Disjoint images of letters means that in fact $w'$ can then only have a single inflation word decomposition induced by any inflation word decomposition of an $N$-neighbourhood.
Hence, $w'$ is recognisable for $\rsub$ with radius of recognisability $N$.
Now, as $w$ contains $\hat{u}$ and $\rsub$ is constant length, that means that $\hat{u}$ forces the possible sequence of cuts for $w$ into level-$1$ inflation words to also be fixed.
Hence, by disjoint images of letters and the fact that $w'$ has a unique decomposition into level-$1$ inflation words given an $N$-neighbourhood, the level-$2$ inflation word decomposition of $w'$ is also unique given an $LN$-neighbourhood.
That is, $w'$ is recognisable for $\rsub^2$ with radius of recognisability $LN$.

To extend the argument to all $p \geq 1$, let $w_{(2)} \coloneqq w$ and $w'_{(2)} \coloneqq w'$, then define $w_{(p)}$ to be a legal word containing $v, \hat{u}, w'_{(2)}, \ldots, w'_{(p-1)}$.
Then $w'_{(p)}$ is chosen to be an appropriate realisation of $\rsub(w_{(p)})$.
The analogous argument to the above then shows that $w'_{(p)}$ is recognisable for $\rsub^p$ with radius of recognisability $L^{p-1}N$.

($\impliedby$)
To prove the right-to-left implication, suppose that $u$ is an $N$-recognisable word for $\rsub^p$ and let $w = u_l u u_r$ be a legal superword of $u$ with $|u_l| = |u_r| = N$.
Suppose that any inflation word decomposition of $w$ for $\rsub^p$ induces the inflation word decomposition $([u_{(0)}, \ldots, u_{(\ell)}],v)$ on $u$.
Without loss of generality, assume that $u \in \rsub^p(v)$.
Then, $u \in \rsub(\rsub^{p-1}(v))$.
As $u$ has a unique decomposition into level-$p$ inflation words induced by $w$ and $\rsub$ has disjoint images of letters (and hence disjoint images by the fact that $\rsub$ is constant length), there is only one word $v' \in \rsub^{p-1}(v)$ such that $u \in \rsub(v')$.
It follows that $w$ induces a unique inflation word decomposition of $u$ for $\rsub$ given by
\[
\left(\left[\rsub\big(v'_0\big), \ldots, \rsub\big(v'_{|v'| - 1}\big)\right],v'\right)
\] for appropriate unique choices of $\rsub(v'_i)$.
\end{proof}
Note that for a constant length substitution with disjoint images of letters, if $u$ is a recognisable word, and we consider a word $\hat{u}$ for which $u$ is a subword, then for some $N$-neighbourhood $\hat{u}_{(N)}$ of $\hat{u}$, any inflation word decomposition $d \in D_\vartheta(\hat{u}_{(N)})$ will induce the same inflation word decomposition on $u$.
In particular, this means that all $\rsub$-cuttings of both $\hat{u}$ and $\hat{u}_{(N)}$ must be the same due to $\rsub$ being constant length.
Then, by the disjoint images property, this implies remarkably that $\hat{u}$ and even $\hat{u}_{(N)}$ have unique inflation word decompositions.
More precisely, we have shown that if $u$ is recognisable with radius of recognisability $N$, then all superwords of $u$ are also recognisable with radius of recognisability $N$, and all sufficiently long superwords (say $w = u^{(l)} u u^{(r)}$ with $|u^{(l)}|, |u^{(r)}| \geq N$) are recognisable with radius of recognisability $0$.

\begin{lemma}\label{LEM:recog-extend}
Let $\rsub$ be a constant length random substitution with disjoint images of letters.
Let $u \in \mc L_\rsub$ be a subword of $w = u^{(l)} u u^{(r)} \in \mc L_\rsub$.
If $u$ is $N$-recognisable, then $w$ is $N$-recognisable.
Further, if $|u^{(l)}|, |u^{(r)}| \geq N$, then $w$ is $0$-recognisable.
\qed
\end{lemma}
This will be used for the main results in this section.

The following technical Lemma will also be needed, allowing us to use a general marker-based method for building automorphisms.

\begin{lemma}\label{LEM:recog-words-big-gaps}
Let $\rsub$ be a non-degenerate, constant length, primitive, compatible random substitution with disjoint images of letters.
If $\rsub$ admits a recognisable word, then for some power $p\geq 1$, $\rsub^p$ admits a recognisable word $u$ such that there is an increasing sequence of lengths $(n_k)_{k \geq 1}$ whereby the sets
\[
V_k \coloneqq \{v \mid |v| = n_k, |uvu|_{u} = 2, uvu \in \mc L_\rsub\}
\]
contain at least $k$ words that share a common unique root.
\end{lemma}
\begin{proof}
Let $u$ be a recognisable word.
Let $p$ be large enough so that $\#\rsub^p(a) \geq 2$ for all $a \in \mc A$ and such that $u \blacktriangleleft \rsub^p(a)$ for some $a \in \mc A$.
Let $a \in \mc A$ be a fixed letter of the alphabet and choose $u_1(a),u_2(a) \in \rsub^p(a)$ such that $u\triangleleft u_1(a)$.
For every other letter $b \in \mc A \setminus \{a\}$, choose a single word $u(b)$.
Now define a new random substitution $\rsub' \colon \mc A \to \mc S$ by $\rsub'(a) = \{u_1(a),u_2(a)\}$ and $\rsub'(b) = \{u(b)\}$ for all $b \neq a$.
By construction, $X_{\rsub'}$ is a non-empty, closed, invariant subspace of $X_\rsub$.
In order to see that $X_{\rsub'}$ is a proper subset of $X_\rsub$, note that Gohlke \cite{G:entropy} showed that the topological entropy $h$ of a compatible random substitution with the `disjoint set condition' is given by a formula just in terms of the associated substitution matrix $M_\rsub$ and the vector $\mathbf{q}\coloneqq (\#\rsub(a_i))_{a_i \in \mc A}$ which is monotonically increasing in the entries of $\mathbf{q}$.
In particular, a constant length random substitution with disjoint images of letters satisfies the disjoint set condition, and so it can easily be shown from Gohlke's formula that $h(X_{\rsub'}) < h(X_\rsub)$ and so $X_{\rsub'}\subsetneq X_\rsub$.

We may now, without loss of generality, assume that $p = 1$.
Let $v \in \mc L_\rsub$ be some legal word which does not appear in the language of $\rsub'$.
It was shown by Rust and Spindeler \cite{RS:random} that $X_\rsub$ contains a point $x$ with a dense orbit.
Both the recognisable word $u$ and the word $v$ appear as subwords of $x$ and so there exists some legal word $\hat{u}$ which contains both $u$ and $v$ as subwords.
In particular, $\hat{u}$ is both recognisable by Lemma \ref{LEM:recog-extend} and does not appear in the language of $\rsub'$.
For a given $k$, let $w_k \in \mc L_{\rsub'}$ be a word sufficiently longer than the word $\hat{u}$ and containing many more than $k$ appearances of the inflation word $u_1(a)$.
By the denseness of the orbit of $x$, there exists a $\rsub$-legal word of the form $y_k = \hat{u} \square w_k \square \hat{u}$ such that $|y_k|_{\hat{u}} = 2$.
Let $n_k \coloneqq |y_k| - 2|\hat{u}|$.

As $w_k$ has many more than $k$ appearances of the inflation word $u_1(a)$, there are at least $2^k$ possible inflation word replacements which can be performed on $w_k$, replacing a subset of the $u_1(a)$s with $u_2(a)$s, to produce a new word $w'_k$.
As these replacements are also legal replacements for $\rsub$, the same replacements in $y_k$ will still produce an element $y'_k$ of $\mc L_\rsub$.
Further, because $w_k$ is sufficiently long, such a set of inflation word replacements can be performed at least a distance $|\hat{u}|$ away from the boundaries between the subword $w_k$ and the unknown blocks $\square$ in $y_k$.
As such, sufficiently many replacements can be performed with no new appearances of $\hat{u}$ being introduced into the new words $y'_k$ because the modified subword $w'_k$ is $\rsub'$-legal.
It follows that $|y'_k|_{\hat{u}} = 2$ for sufficiently many new words $y'_k$ that are obtained from $y_k$ by an inflation word replacement of this kind.
Also, for each of these words $y'_k$, $|y'_k| - 2|\hat{u}| = n_k$ by compatibility of $\rsub$, as required.
\end{proof}

In order to make up for the fact that we do not have recognisability of the random substitution, we will use the word $u$ constructed in Lemma \ref{LEM:recog-words-big-gaps} as a `marker' between which we are allowed to apply inflation word replacements.
Further, as we can control the distances between these gaps, thanks to the fact that the words in the sets $V_k$ all have the same length $n_k$, it means that such inflation word replacements are applied at mutually distinct locations in an element $x\in X_\rsub$, and so commute with one another.
This allows us to mimic the classic marker-based argument introduced by Hedlund \cite{H:curtis-hedlund-lyndon} and generalised by Boyle--Lind--Rudolph \cite{BLR:auto-sft} in order to embed an arbitrary countable direct sum of finite groups into the automorphism group $\Aut(X_\rsub)$ when the random substitution is not necessarily recognisable but is constant length, has disjoint images and admits a recognisable word (as the random period doubling substitution does).
\begin{theorem}\label{THM:const-length-fin-groups}
Let $\rsub$ be a non-degenerate, constant length, primitive, compatible random substitution with disjoint images of letters and which admits at least one recognisable word $u$.
Let $G$ be a countable direct sum of finite groups.
The group $G$ embeds into the automorphism group $\Aut(X_\rsub)$.
\end{theorem}
\begin{proof}
Let $p \geq 1$ and $u \in \mc L_\rsub$ be such that $u$ is recognisable with respect to $\rsub^p$ and there is an increasing sequence of lengths $(n_k)_{k \geq 1}$ whereby the sets
\[
V_k \coloneqq \{v \mid |v| = n_k, |uvu|_{u} = 2, uvu \in \mc L_\rsub\}
\]
contain at least $k$ words that share a common unique root, as guaranteed to exist by Lemma \ref{LEM:recog-words-big-gaps}.
Let $k \geq 1$ be given.
Our goal is to show that the symmetric group $S_k$ embeds into $\Aut(X_\rsub)$.

Let $v_1, \ldots, v_k$ be $k$ elements in the set $V_k$.
For a permutation $\pi \in S_k$, we define an automorphism $f_{\pi} \colon X_\rsub \to X_\rsub$ by mapping an element $x \in X_\rsub$ to an inflation word replacement $f_\pi(x)$ given by replacing any occurrence of $u v_i u$ in $x$ with the word $u v_{\pi(i)} u$.
The map $f_\pi$ is well-defined, as the recognisable word $u$ only appears as a subword of $uv_ku$ as a prefix and suffix, and nowhere else.
Hence, no two words $uv_{k}u$ and $uv_{\ell}u$ can overlap in $x$ except as a prefix-suffix pair.
As $u$ remains unchanged under the map, there is no ambiguity in where inflation words are being replaced.
As $f_\pi$ is defined locally, it is continuous, and shift-invariant.
It is clear that $f_\pi \circ f_{\pi^{-1}} = f_{\pi^{-1}} \circ f_\pi = \operatorname{Id}_{X_\rsub}$ and so $f_\pi$ is an automorphism.
Likewise, the map $S_k \to \Aut(X_\rsub) \colon \pi \mapsto f_\pi$ is a homomorphism, and as $v_k = v_{\pi(k)}$ for all $k$ if and only if $\pi = e$, this mapping is also an embedding.

Write $H_k = \{f_\pi \mid \pi \in S_k\}$.
We note, by construction, that if $k \neq \ell$, then $n_k \neq n_\ell$ and so if $f_\pi \in H_k$ and $f_\rho \in H_\ell$, then $\left[f_\pi,f_\rho\right] = \operatorname{Id}_{X_\rsub}$.
That is, if $k \neq \ell$, then automorphisms in $H_k$ commute with automorphisms in $H_\ell$.
So $\bigoplus_{k=1}^\infty H_k$ is a subgroup of $\Aut(X_\rsub)$.
Hence, by Cayley's theorem, $G$ embeds into $\Aut(X_\rsub)$ because $\bigoplus_{k=1}^\infty H_k$ does.
\end{proof}

A similar marker-based technique, again in the style of Hedlund and Boyle--Lind--Rudolph, can be used to show that the free product $(\Z/2\Z)^{\ast k}$ of $k$ involutions also embeds into $\Aut(X_\rsub)$, and hence any free group $F_k$ on $k \geq 1$ generators, and the free group $F_\omega$ on countably many generators.
We omit this proof, as the method is now well-established.
\begin{theorem}\label{THM:const-length-freegroup}
Let $\rsub$ be a constant length, primitive, compatible random substitution with disjoint images of letters and which admits at least one recognisable word $u$.
Let $k \geq 1$ and let $G = (\Z/2\Z)^{\ast k}$.
The group G embeds into the automorphism group $\Aut(X_\rsub)$.
\end{theorem}
In fact, the marker method can be generalised to show that the automorphism group $\Aut(X_n)$ of any full shift embeds into $\Aut(X_\rsub)$ when the conditions on $\rsub$ in Theorem \ref{THM:const-length-fin-groups} are satisfied.
Theorems \ref{THM:const-length-fin-groups} and \ref{THM:const-length-freegroup} are then simple corollaries of this stronger embedding result.

The idea is that for $v_0,v_1 \in V_2$, we use $uv_iu$ as a marker to code elements of the full $2$-shift over the symbols $\{0,1\}$.
As the words $uv_iu$ appear with uniformly bounded gaps in some element $x \in X_\rsub$, then an automorphism on the $2$-shift induces a local map on the coding words $v_i$.

The tricky thing is what one then does for elements that do not have the words $uv_i u$ appearing with uniformly bounded gaps.
What does one do when reaching the `end' of a run of $uv_iu$s that is followed by a large gap before the next occurrence (if any)?
To get around this difficulty, we use a slightly more involved encoding of a full shift which carries two bits of information, rather than one.
The induced action of an automorphism on a full shift then acts on the `top' bit of information as usual, but on the `bottom' bit of information in reverse.
At the end of a run, you then `wrap' this action around.
This idea is sometimes referred to as the \emph{conveyor belt} method \cite{S:conveyor}.

\begin{theorem}\label{THM:const-length-full-shift}
Let $\rsub$ be a constant length, primitive, compatible random substitution with disjoint images of letters and which admits at least one recognisable word $u$.
Then $\Aut(X_\rsub)$ contains the automorphism group $\Aut(X_2)$ of the full $2$-shift $X_2$.
\end{theorem}
\begin{proof}
By Lemma \ref{LEM:recog-words-big-gaps}, and replacing $\rsub$ with a power if necessary, there exists a word $u \in \mc L_\rsub$ such that $u$ is recognisable with respect to $\rsub$ and there is a length $n_4$ whereby the set
\[
V_4 = \{v \mid |v| = n_4, |uvu|_u = 2, uvu \in \mc L_\rsub\}
\]
contains at least $4$ words that share a common unique root.
Label four of these words $v_0^0, v_1^0, v_0^1, v_1^1$.
In this coding, we think of the superscript as referring to the  `top track', and the subscript as being the `bottom track' of a conveyor belt.

Let $a \in \mc A$ and $p \geq 1$ be such that $uv_i^ju \blacktriangleleft\rsub^p(a)$ for all $(i,j) \in \{0,1\}^2$.
As $X_\rsub$ has uniformly bounded gaps between letters, let $\hat{N} \geq 1$ be such that every legal word in $\mc L^{\hat{N}}_\rsub$ contains at least two copies of the letter $a$.
Let $N = L^p\hat{N}$.
Then, by taking all realisations of $\rsub^p(a)$ to contain $uv_0^0u$, there exists an element $x_\infty \in X_\rsub$ such that every subword $w \triangleleft x_\infty$ of length at least $|w| \geq N$ contains two copies of $uv_0^0u$.

For $x \in X_\rsub$, write
\[
x = \cdots \blacksquare \: uv_{i(n)}^{j(-n)}u \: \square \:  \cdots  \: \square \: uv_{i(1)}^{j(-1)}u \: \square \: uv_{i(0)}^{j(0)}u \: \square \: uv_{i(-1)}^{j(1)}u \: \square \:  \cdots \: \square \: uv_{i(-m)}^{j(m)}u \: \blacksquare \cdots,
\]
where $\square$ represents a word such that $u\square u$ contains no appearances of a word $uv_i^ju$ with $(i,j) \in \{0,1\}^2$ with length $|\square| \leq N$, and $\blacksquare$ represents a word such that $u\blacksquare u$ contains no appearance of a word $uv_i^ju$ with $(i,j) \in \{0,1\}^2$ with length $|\blacksquare| >N$.
Note that $\blacksquare$ can possible be a left(right)-infinite word.
Notice that the indices $i$ of the bottom track are reversed.
For ease of notation, we will suppress the markers $u$ in this representation, being absorbed into a neighbouring $\square$ or $\blacksquare$, and so write
\[
x = \cdots \blacksquare \: v_{i(n)}^{j(-n)} \: \square \:  \cdots  \: \square \: v_{i(1)}^{j(-1)} \: \square \: v_{i(0)}^{j(0)} \: \square \: v_{i(-1)}^{j(1)} \: \square \:  \cdots \: \square \: v_{i(-m)}^{j(m)} \:\blacksquare \cdots.
\]

Let $X_2$ be the full $2$-shift.
Let $\alpha$ be an automorphism of the full $2$-shift $X_2$.
We define an automorphism $f_\alpha \colon X_\rsub \to X_\rsub$ on $x$ in the following way.

If $x$ contains no $\blacksquare$, then 
\[
f_\alpha(x) \coloneqq \cdots \: \square \: v_{\alpha(i)(2)}^{\alpha(j)(-2)}  \: \square \: v_{\alpha(i)(1)}^{\alpha(j)(-1)} \: \square \: v_{\alpha(i)(0)}^{\alpha(j)(0)} \: \square \: v_{\alpha(i)(-1)}^{\alpha(j)(1)} \: \square \: v_{\alpha(i)(-2)}^{\alpha(j)(2)} \: \square \: \cdots,
\]
where $i$ and $j$ are considered as elements of the full $2$-shift $X_2$.
So we have applied $\alpha$ to the top indices of the $v_{i}^{j}$ and also to the bottom indices of $v_{i}^{j}$ but where the sequence $i$ is in reverse.

If $x$ contains a $\blacksquare$, then near this position, write

\[
x = \cdots v_{i(2)}^{i(-3)} \: \square \: v_{i(1)}^{i(-2)} \: \square \: v_{i(0)}^{i(-1)} \: \blacksquare \: v_{j(-1)}^{j(0)} \: \square \: v_{j(-2)}^{j(1)} \: \square \: v_{j(-3)}^{j(2)} \: \square \: \cdots ,
\]
where $i$ and $j$ are considered as elements of the full $2$-shift that have been `wrapped over' when they get near a $\blacksquare$.
Then we define

\[
f_\alpha(x) \coloneqq \cdots v_{\alpha(i)(2)}^{\alpha(i)(-3)} \: \square \: v_{\alpha(i)(1)}^{\alpha(i)(-2)} \: \square \: v_{\alpha(i)(0)}^{\alpha(i)(-1)} \: \blacksquare \: v_{\alpha(j)(-1)}^{\alpha(j)(0)} \: \square \: v_{\alpha(j)(-2)}^{\alpha(j)(1)} \: \square \: v_{\alpha(j)(-3)}^{\alpha(j)(2)} \: \square \: \cdots.
\]

Between two appearances of $\blacksquare$, the indices of the $v_i^j$ act like the coding of a periodic element of $X_2$ wrapped around in a circle.

As everything is locally defined, $f_\alpha$ is well-defined, continuous and commutes with the shift action.
It is also clear that $f_{\alpha^{-1}} = f_\alpha^{-1}$ and $f_{\alpha\circ \beta} = f_\alpha\circ f_\beta$, and so $f_\alpha$ is an automorphism.
It remains to check that the map $\alpha \mapsto f_\alpha$ is faithful.

Let $\alpha$ be non-trivial.
So there exists an element $j \in X_2$ such that $\alpha(j) \neq j$.
Take an element $x \in X_\rsub$ with no appearances of a $\blacksquare$ with the form
\[
x = \cdots \: \square \: v_{j(2)}^{j(-2)}  \: \square \: v_{j(1)}^{j(-1)} \: \square \: v_{j(0)}^{j(0)} \: \square \: v_{j(-1)}^{j(1)} \: \square \: v_{j(-2)}^{j(2)} \: \square \: \cdots,
\]
which exists because the element $x_\infty$ constructed above exists (just replace the appearances of $v_0^0$ with a corresponding $v_{j(-n)}^{j(n)}$).
By construction, $f_\alpha(x) \neq x$ and so $f_\alpha$ is non-trivial.
Hence the map $\alpha \mapsto f_\alpha$ is faithful and so $\Aut(X_2)$ embeds in $\Aut(X_\rsub)$.
\end{proof}

\begin{coro}
Let $\rsub \colon a \mapsto \{ab,ba\}, b\mapsto \{aa\}$ be the random period doubling substitution and let $X_\rsub$ be its RS-subshift.
The automorphism group $\Aut(X_\rsub)$ contains the automorphism group $\Aut(X_2)$ of the full $2$-shift $X_2$. \qed
\end{coro}
\subsection{Discussion and Open Questions}
One is naturally drawn to comparing random substitution subshifts with their deterministic cousins. On the other hand, our study of automorphism groups for RS-subshifts shows that they compare more closely with those of full shifts and shifts of finite type, having rather large and `wild' automorphism groups, while deterministic substitution subshifts have automorphism groups that are finite extensions of $\mathbb{Z}$. This is not so surprising given the plethora of recent results that highlight the connections between low complexity and small automorphism groups \cite{PS:autos-low-complexity}, coupled with the fact that substitution subshifts have linear complexity (in particular zero entropy), while the RS-subshifts studied here all have exponential complexity (positive entropy).

One important similarity between our methods and previous studies of deterministic substitutions is the importance of recognisability \cite{BSTY:recog} as a tool for understanding and manipulating the subshifts---indeed, the hierarchical structure of the subshift is exactly what allows for the shuffle group to be defined and which gives rise to this interesting subgroup of automorphism on the RS-subshift. When recognisability is no longer available, we have seen in Section \ref{SEC:const-length} that more traditional `marker methods' from the world of SFTs are needed to prove similar results on the automorphism group.

We have only provided here a first investigation of these intricate automorphism groups and their interplay with the particular properties of the defining random substitution.
There is still much to be understood about the subgroup structure and other open questions.

While the random period doubling substitution admits periodic elements in its subshift, not all of the random substitutions satisfying the conditions of Theorem \ref{THM:const-length-full-shift} have periodic points.
Therefore, it is an open question as to whether these RS-subshifts have residually finite automorphism groups when they are aperiodic and non-recognisable.

For example, consider the random substitution $\rsub \colon a \mapsto \{abaa, aaba\}, \: b \mapsto \{abab, baba\}$.
It is constant length, primitive, compatible and admits the recognisable word $aaaa$.
It is not recognisable because the two marginal substitutions $\sub_1 \colon a \mapsto abaa,\: b \mapsto abab$ and $\sub_2 \colon a \mapsto aaba,\: b \mapsto baba$ are conjugate ($(aba)^{-1}\sub_1(aba) = \sub_2$), and therefore their bi-infinite fixed points are the same up to a shift, hence admit two preimages.
Nevertheless, $X_\rsub$ is aperiodic, which can be seen by verifying the criterion given in \cite[Corollary 36]{R:random-periodic}, as every element of $X_\rsub$ must contain at least one of $aaab$, $baaa$, $abba$ or $baab$ as a subword.
We therefore do not currently know if $\Aut(X_\rsub)$ is residually finite or not.
Most likely the existence of a recognisable word is still enough for the automorphism group to contain an infinite simple subgroup. So, as well as the above, we ask the following more general question.
\begin{question}\label{Q:aperiodic-res-fin}
Does there exist a primitive, compatible, non-recognisable random substitution $\rsub$ whose subshift is aperiodic but whose automorphism group \textbf{(i)} is residually finite, \textbf{(ii)} does not contain an infinite simple subgroup?
\end{question}
An example satisfying Question \ref{Q:aperiodic-res-fin} would be interesting, as we do not yet have any examples of random substitutions whose subshifts can be differentiated up to topological conjugacy by comparing their automorphism groups, other than by using much simpler invariants such as the existence of periodic points.
Perhaps, when $\rsub$ is recognisable, the shuffle group $\Gamma$ is a good candidate for differentiating pairs of examples, as we have a much better understanding of the structure of $\Gamma$.
\begin{question}
Is $\Gamma$ (or a characteristic property of $\Gamma$ such as some notion of asymptotic growth rate) an invariant of topological conjugacy for $X_\rsub$?
\end{question}
This gives rise to the more philosophical question:
\begin{question}
How good is $\Aut(X_\rsub)$ at classifying RS-subshifts?
\end{question}

As far as more general properties of the automorphism group are concerned, natural questions that we have not yet explored include whether $\Aut(X_\rsub)$ is finitely generated, whether Ryan's Theorem \cite{R:Ryans-thm} holds for $\Aut(X_\rsub)$, and whether the shift can admit finite roots.
\begin{question}
Under what conditions is $\Aut(X_\rsub)$ finitely generated.
\end{question}
\begin{question}
Is the center $Z(\Aut(X_\rsub))$ always generated by the shift $\sigma$?
\end{question}
\begin{question}
In $\Aut(X_\rsub)$, when does the shift map $\sigma$ have a finite root? That is, when does there exist $f \in \Aut(X_\rsub)$ and $n \geq 2$ such that $f^n = \sigma$?
\end{question}
A rather artificial example of an automorphism whose square is the shift is given by considering the non-primitive random substitution
\[
\rsub \colon a \mapsto \{\overline{ab},\overline{ba}\}, \overline{a} \mapsto \{ab,ba\}, b \mapsto \{\overline{aa}\}, \overline{b} \mapsto \{aa\}.
\]
Every element of the RS-subshift either has all barred letters or no barred letters, so we can define an automorphism $f$ that removes bars on barred letters and for non-barred letters, adds bars and shifts the sequence. Then $f^2 = \sigma$.

\section*{Acknowledgement}
The authors are grateful to an anonymous referee for pointing out an error in a previous version of this work.
The authors would like to thank Raf Bocklandt, Philipp Gohlke, Neil Ma\~{n}ibo, Eden Miro, Samuel Petite and Scott Schmieding for helpful discussions.
We are especially thankful to Scott Schmieding for alerting us to other examples of subshifts with automorphism groups containing an infinite simple subgroup.
Dan Rust would like to acknowledge the support of the DFG via SFB1283/1 and the support of the Dutch Science Federation (NWO) through visitor grant 040.11.700.

\bibliographystyle{jis}
\bibliography{tilings}

\end{document}